\documentclass{article}
\usepackage[utf8]{inputenc}
\usepackage{graphicx}%
\usepackage{multirow}%
\usepackage{amsmath,amssymb,amsfonts}%
\usepackage{amsthm}%
\usepackage{mathrsfs}%
\usepackage[hidelinks]{hyperref}
\usepackage[title]{appendix}%
\usepackage{xcolor}%
\usepackage{textcomp}%
\usepackage{manyfoot}%
\usepackage{booktabs}%
\usepackage{algorithm}%
\usepackage{algorithmicx}%
\usepackage{algpseudocode}%
\usepackage{listings}%
\usepackage{subcaption}
\usepackage{cleveref}
\usepackage{tabularx}
\usepackage{todonotes}
\usepackage[T1]{fontenc}
\usepackage[margin=1in]{geometry}
\usepackage{fancyhdr}
\usepackage{placeins}  
\usepackage{float}      
\usepackage{tabularx}   
\usepackage{ulem}

\def\diag{\textnormal{diag}}

\def\id{\textnormal{I}}
\newcommand{\interleave}{\large|\hspace{-0.1em}\large|\hspace{-0.1em}\large|}

\usepackage{lipsum} 

\newcommand{\dy}[1]{\textcolor{brown}{#1}}

\newtheorem{theorem}{Theorem}
\newtheorem{prop}[theorem]{Proposition}%
\usepackage{geometry}
\newtheorem{remark}{Remark}%
\newtheorem{problem}{Problem}
\newtheorem{lemma}{Lemma}

\newtheorem{definition}{Definition}%
\usepackage[inline,shortlabels]{enumitem}

\title{Random space-time sampling and reconstruction of sparse bandlimited graph diffusion field}

\begin{document}


\author{
  Longxiu Huang\thanks{Department of CMSE and Mathematics, Michigan State University. Email: huangl3@msu.edu} 
  \and
  Dongyang Li\thanks{Department of Mathematics, University of California Santa Barbara. Email: dongyangli@ucsb.edu} 
  \and
  Sui Tang\thanks{Department of Mathematics, University of California Santa Barbara. Email: suitang@ucsb.edu} 
  \and
  Qing Yao\thanks{Department of Linguistics, University of Texas at Austin. Email: qyao@utexas.edu}
}

\date{}

\maketitle

\abstract{In this work, we investigate the  sampling and reconstruction of spectrally $s$-sparse bandlimited graph signals governed by heat diffusion processes. We propose a random space-time sampling regime, referred to as {randomized} dynamical sampling, where a small subset of space-time nodes is randomly selected  at each time step based on a  probability distribution.  To analyze the recovery problem, we establish a rigorous mathematical framework  by introducing  the parameter \textit{the dynamic spectral graph weighted coherence}. This key parameter governs the number of space-time samples needed for stable recovery and extends the idea of variable density sampling to the context of dynamical systems. By optimizing the sampling probability distribution, we show that as few as $\mathcal{O}(s \log(k))$ space-time samples are sufficient for accurate reconstruction in optimal scenarios, where $k$ denotes the bandwidth of the signal. Our framework encompasses both static and dynamic cases, demonstrating a  reduction in the number of spatial samples needed at each time step by exploiting temporal correlations. Furthermore, we provide a computationally efficient and robust algorithm for signal reconstruction. Numerical experiments validate our theoretical results and illustrate the practical efficacy of our proposed methods. }

\section{Introduction}
Graph signals, which model complex relationships across various domains such as social networks and transportation systems, have garnered increasing attention in recent years, as evidenced by studies from Shuman et al. \cite{shuman2012graphs, shuman2013emerging,sandryhaila2014big,ortega2018graph,jian2024kernel}. Processing data from large-scale graphs presents significant computational and storage challenges,  highlighting the need for efficient data compression techniques that preserve essential information \cite{tanaka2020sampling}. As a result, extensive research has been dedicated to optimizing the sampling and reconstruction of graph signals to ensure accurate recovery from discrete samples \cite{pesenson2008sampling,chen2015discrete,pesenson2009variational,segarra2016reconstruction,puy2018random,shen2024reconstruction}. 

The problems of sampling and reconstruction are inherently ill-posed without additional constraints on graph signals. Significant advancements in the sampling theory for graph signals \cite{anis2016efficient, puy2018random, tanaka2020sampling} have introduced low-dimensional structural assumptions, such as smoothness \cite{kalofolias2016how,dong2016learning,giraldo2022reconstruction} and spectral sparsity \cite{tropp2009beyond}, which extend classical sampling theory to regular domains. These assumptions, supported by empirical evidence from real-world graph signals, facilitate effective graph signal compression through the proposed sampling algorithms  \cite{dynamic_sampling,chen2015discrete}. 

In contemporary applications, graph signals frequently exhibit time-dependence, as seen in phenomena such as rumor propagation on social networks \cite{xiao2019rumor} and neural activity transfer in the brain \cite{sporns2010networks}. However, much of the existing work does not account for this temporal dependence. This time-dependence adds complexity to the sampling of graph signals that evolve over time, particularly when resources are limited and sufficient samples cannot be taken at a single time instance. This makes traditional sampling theory for static graph signals insufficient. A pertinent example is electrocorticography (ECoG), where spatial constraints limit the number of usable devices at any given time.   Nevertheless, the temporal dependence among graph signals offers new possibilities: the deficiency of spatial information at a particular time can be compensated by aggregating data across multiple instances, suggesting a shift towards sparse space-time sampling strategies.

This paper addresses the problem of sampling and reconstructing dynamical graph signals, modeled using affine linear dynamical systems—a problem we refer to as the \textit{dynamical sampling problem}. This is a well-known and challenging inverse problem, as it involves reverse the heat diffusion process from partially observed dynamical data. We focus on random sparse space-time sampling and establish sampling theories with corresponding recovery guarantees. Additionally, we employ $\ell_1$ regularization in the reconstruction algorithm, exploiting the spectral sparsity of the initial state of the dynamical system. Our work extends the theory of spectrally sparse signals in the static case and generalizes compressive sensing algorithms to the dynamical setting.

\paragraph{Low-Dimensional Structure in Dynamic Graph Data}

We denote the signals at \(t\in \{0\}\cup [T-1]\) by \(\mathbf{x}_t \in \mathbb{R}^n\) for some \(T \in \mathbb{Z}_+\). Dynamical systems often model time-dependent signals. For simplicity, we initially focus on the heat diffusion process governed by the equation 
\begin{equation}\label{eqn:DS}
\frac{\partial }{\partial t} \mathbf{x}_t = - \mathbf{L} \mathbf{x}_t, \quad t \geq 0,
\end{equation}
where \(L\) is the graph Laplacian. The  solution to \eqref{eqn:DS} is \(\mathbf{x}_t = e^{-t\mathbf{L}} \mathbf{x}_0\). For discrete time steps \(\{t \cdot \Delta t\}_{t=0}^{T-1}\) with \(\Delta t > 0\), we define the signal evolution operator as \(\mathbf{A} = e^{-\Delta t \mathbf{L}}\), yielding \(\mathbf{x}_t = \mathbf{A}^t \mathbf{x}_0\).

The graph heat diffusion equation models natural phenomenon such as   temperature variations, air pollution dispersion, and photon transportation in tissues where the initial state is smooth and band-limited \cite{lu2009distributed}. The framework developed in this paper extends to more general dynamics beyond heat diffusion and can also deal with irregular time instances, as discussed in Section \ref{sec:generalize}.

In practice, the bandwidth of the signal is often unknown and assumed to be large. In such cases, the initial state has sparse support in the frequency domain, which is the focus of our paper. The concept of sparsity in the frequency domain, evident in the spectrally sparse vertex representation of many real-world graph signals, extends the Euclidean domain sparsity assumptions discussed by Tropp et al. \cite{tropp2009beyond}.

In graph signal analysis, prevalent methods often depend on assumptions such as bandlimitedness \cite{anis2016efficient}, Sobolev smoothness \cite{giraldo2022reconstruction}, or smooth temporal differences \cite{qiu2017time}. This paper models smooth graph signals as sparse within a substantial bandwidth in the graph Fourier domain, where these signals are also described by a linear dynamical system.

  \paragraph{Space-time Sampling and Reconstruction Problem.}
In dynamical sampling problems, the recoverability of the initial signal \(\mathbf{x}_0\) typically depends on the spectral properties of the evolution operator \(A\) and the structure of the space-time sampling. Denote the sampling set at time step \(t\in \{0\}\cup [T-1]\) by \(\Omega_t = \{\omega_{t, 1}, \ldots, \omega_{t, m_t}\} \subset [n]\), which is randomly selected  according to some probability distribution on \([n]\). The corresponding sampling operator \(\mathbf{S}_{\Omega_t} \in \mathbb{R}^{m_t \times n}\) is defined by
\[
\mathbf{S}_{\Omega_t}(i,j) = 
\begin{cases}
    1, & \text{ if } j = \omega_{t, i} \\
    0, & \text{ otherwise}
\end{cases}.
\]
Our space-time reconstruction problem is described as follows:

\begin{problem}\label{prob:space-time}
    Let \(\mathbf{x}\) be \(s\)-sparse and \(k\)-bandlimited with a space-time trajectory given by \(\mathbf{x}_0 = \mathbf{x}, \mathbf{x}_1, \ldots, \mathbf{x}_{T-1}\). At each time step $t \in \{0\}\cup [T-1]$, we have  a space-time sampling set \(\Omega_t \subset [n]\). The goal is  to recover \(\mathbf{x}\) from its space-time samples \(\{\mathbf{y}_t = \mathbf{S}_{\Omega_t} \mathbf{x}_t : t\in \{0\}\cup [T-1]\}\), where \(\mathbf{S}_{\Omega_t}\) is the sampling operator that selects the entries indexed by \(\Omega_t\).
\end{problem}

In the static setting where $T=0$, sampling graph signals is equivalent to selecting nodes on the graph. However, in the dynamic case, we may have more sampling modalities adapted to practical constraints. For example, we may only have a very limited number of sensors to use per time step but are allowed the flexibility to update our sampling locations at each time iteration. This motivates our focus on the following sampling regime:

\noindent\textit{Random selections at each iteration}: At each time step \(t\in \{0\}\cup [T-1]\), we randomly sample \(m_t\) nodes according to some  probability distribution. Mathematically, let \(\mathbf{p}_t \in \mathbb{R}^n\) represent the probability distribution over the nodes \([n]\) for each time step \(t\in \{0\}\cup [T-1] \). The sampling sets \(\Omega_t\) are selected independently with replacement at each time step, according to \(\mathbf{p}_t\). This regime models scenarios where sensors can be updated or changed over time. In this work, our main goal is to address the following two questions:

\begin{enumerate}[label=(\roman*),leftmargin=.25in]
    \item \label{main-i} What conditions should  $\mathbf{p}_t$ and $m_t$  satisfy to ensure that  Problem~\ref{prob:space-time} can be  solved? How should the sampling probabilities be chosen to be compatible with the graph structure and dynamics to optimize data recovery?
    \item \label{main-ii} What algorithms are suitable for the reconstruction of \(\mathbf{x}\) from its space-time samples?
\end{enumerate}

    \paragraph{Main results.} 

To address Question \ref{main-i}, we introduce a novel graph coherence parameter, $\nu_{k,T}$, which is defined by incorporating dynamics, sparsity, and spatial sampling patterns. Note that this parameter  extends   the static graph coherence defined in \cite{puy2018random}  for bandlimited graph signal.  
Through appropriate reweighing,  we show that the sampling complexity required to ensure the restricted isometry property (RIP) of the sampling matrices for the  target graph signals is proportional to $\left(\nu_{k,T}\right)^2 s \log \frac{k}{s}$. The  linear dependence on $s$ is validated by the numerical experiments (see Figure~\ref{fig:rec_r} and Figure~\ref{fig:rec_m}). Notably,  $\nu_{k,T}^2$ can be as small as $o(1)$  (see Figure~\ref{tab:coh_r}) for regular graphs such as the ring graph. Our approach leverages an embedding technique that reformulates the dynamic problem as a static sparse recovery problem within an extended space-time domain. Technically, we handle a structured random sensing matrix governed by the underlying physics, which differs from traditional sensing matrices with i.i.d. entries commonly used in compressed sensing. Identifying the appropriate incoherence that reflects the intrinsic sampling complexity for this problem presents a non-trivial challenge.

Furthermore, we derive optimal sampling distributions by minimizing the coherence parameters, resulting in an adaptive space-time sampling technique that leverages temporal dynamics to optimize sampling complexity in the proposed data recovery problems.

To address Question~\ref{main-ii}, we leverage the RIP of our re-weighted sampling operators to enable robust signal reconstruction via $\ell_1$-minimization. Specifically, we implement the CoSaMP algorithm and provide the corresponding  recovery theorems. Beyond the heat diffusion process, we extend our methods to broader classes of linear dynamics in \Cref{sec:generalize}. In \Cref{sec:experiments}, we validate our theoretical results on sampling complexity and demonstrate the robustness of our algorithms through various numerical examples.

\section{Related Works}
Our work lies at the intersection of compressed sensing and dynamic graph signal processing. 

\paragraph{Dynamical Sampling} Our work builds on the dynamical sampling framework introduced by Aldroubi and his collaborators \cite{aldroubi2013dynamical,aldroubi2017dynamical}, which aims to provide a fundamental mathematical understanding of space-time sampling for dynamical systems. While most   studies focus on deterministic settings and the feasibility of sampling sets, a less-explored topic, is how to determine optimal sampling locations and times based on physical priors to enhance data recovery. We extend recent advances in coherence-guided randomized dynamical sampling, as explored in \cite{dynamic_sampling}, to address this challenge.

The work in \cite{dynamic_sampling} investigates dynamical sampling of bandlimited graph signals, primarily relying on \(\ell_2\)-regularization for signal recovery. In contrast, this work emphasizes the compressibility of graph signals and uses \(\ell_1\)-minimization for recovery. This paper extends our previous work in \cite{conference_article}, where we addressed a similar problem but selected space-time samples using a probability distribution over \([n] \times \{0,1,\cdots,T-1\}\). In this paper, we explore a different regime, where randomized spatial samples are taken at each time step \(t \in \{0\}\cup [T-1]\). 

\paragraph{Dynamic Graph Signal Processing}
The reconstruction of time-varying graph signals from space-time samples has gained significant research interest in recent years, as documented in works such as \cite{wang2015distributed,qiu2017time,jiang2020,giraldo2022reconstruction,isufi2020observing,lu2021probabilistic,romero2017kernel, jian2023kernel,guo2024time}. These studies typically rely on assumptions about graph dynamics, like bandlimited graph processes \cite{wang2015distributed,chen2015signalieee} or smooth temporal differences \cite{qiu2017time,giraldo2022reconstruction}, or the signal class, such as bandlimitedness \cite{dynamic_sampling, puy2018random}, to facilitate recovery through various minimization problems \cite{wang2015distributed,qiu2017time,giraldo2022reconstruction}. While these methods have shown empirical success, a comprehensive mathematical framework for optimal data acquisition in dynamical systems for exact and robust recovery remains underdeveloped.

In this work, we focus on affine-linear dynamics and leverage signal sparsity, enabling recovery through $\ell_1$-minimization while providing restricted isometry property (RIP) guarantees on the sampling operators. A recent kernel-based reconstruction approach for generalized graph signal processing \cite{jian2023kernel} closely aligns with our setting, as it prioritizes critical nodes for recovery and minimizes an error term combined with a regularization term. However, our model leverages an additional sparsity structure not considered in the kernel methods, further enhancing the efficacy of sparsity-minimizing recovery techniques. Unlike methods that approximate functions from the full space of functions $F(V \times\{0,1,\cdots,T-1\}, \mathbb{R})$, where $V$ denotes vertices and $T$ is the maximum time iteration, our recovery algorithm operates within a subspace characterized by 
\begin{equation}
\left\{f_{(c_1, \cdots, c_k)}(v, t) = \sum_{i=1}^{k} c_i \lambda_i^t \mathbf{u}_i(v) : (c_1, \cdots, c_k) \in \mathbb{R}^k \right\},
\end{equation}
based on our assumptions on graph dynamics.

Our dynamic assumptions are not overly restrictive. For example, in \cite{wei2019optimal}, the authors discuss a deterministic sampling procedure for initially bandlimited (possibly nonlinear) dynamic graph signals, assuming Lyapunov stability—a condition met by many real-world dynamic signals. This supports the notion that nonlinear processes can be approximately linearized, suggesting that our focus on affine-linear dynamical processes is both reasonable and not unduly limiting.

\paragraph{Compressed Sensing and $\ell_1$-Minimization}
In classical signal processing, the sampling of sparse and bandlimited signals has been extensively studied \cite{tropp2009beyond}. A key finding is that for signals with bandwidth $k$ and sparsity level $s \ll k$, approximately $\mathcal{O}(s \log \frac{k}{s})$ samples are required for stable recovery \cite{tropp2009beyond}, representing a significant reduction compared to the traditional Nyquist sampling rate. However, the sampling theory for signals that are sparse in the graph Fourier transform domain remains underdeveloped in graph settings. In our recent work \cite{conference_article}, we employ an $\ell_1$-minimization approach for recovering dynamic graph signals, although the sampling methods are limited to the diffusion process.

During the sampling phase, we utilize a random variable density sampling strategy \cite{puy2011vds, lustic2007mri} adapted to dynamic graphs. The graph coherence parameters we introduced are analogous to those controlling recovery conditions in static contexts \cite{conference_article}, but extending these theorems to dynamic graph signals necessitates further theoretical work.

For the recovery process, we implemented the CoSaMP algorithm \cite{cosamp} due to its simplicity in terms of parameters and implementation. Other  algorithms include LASSO \cite{tibshirani_lasso}, OMP \cite{TroppOMP}, and IHT \cite{blumensath2009iterative} could also be suitable for  the recovery process. 

Below, we provide a pseudo-code for the CoSaMP algorithm, as described in \cite{Foucart2013-sv}. In the algorithm, $H_s$ denotes the hard-thresholding operator of order $s$. For any vector $\mathbf{c}$, $H_s$ retains the $s$ largest entries in absolute value  and sets all other entries to zero.

\begin{algorithm}
\caption{Compressive Sampling Matching Pursuit (CoSaMP) for solving $\Phi \mathbf{c} =\mathbf{y}$}\label{alg:cosamp}
\begin{algorithmic}
\State \textbf{Input:} Measurement matrix: $\Phi$;  Measurement vector: $\mathbf{y}$; Sparsity level: $s$
\State \textbf{Initialization:} $\mathbf{c}^0 = \mathbf{0}$
\While{Halting criterion is not met at $n = \overline{n}$}
\State $V^{n+1} = \operatorname{supp}(\mathbf{c}^n) \cup L_{2s}(\Phi^*(\mathbf{y} - \Phi \mathbf{c}^n))$
\State $\mathbf{v}^{n+1} = \underset{\mathbf{z} \in \mathbb{C}^N}{\operatorname{argmin}}\{\|\mathbf{y} - \Phi \mathbf{z}\|_2, \operatorname{supp}(\mathbf{z}) \subset V^{n+1}\}$
\State $\mathbf{c}^{n+1} = H_s(\mathbf{v}^{n+1})$
\EndWhile
\State \textbf{Output:} $s$-sparse vector $\mathbf{c}^{\sharp} = \mathbf{c}^{\bar{n}}$
\end{algorithmic}
\end{algorithm}

\section{Preliminaries and Notations}\label{sec:intro}
Throughout the paper, column vectors are in bold and {linear operators are also in bold and denoted by uppercase letters}. For any $n \in \mathbb{Z}_{>0}$, we denote $[n] = \{1, 2, \cdots, n\}$. For any vector $\mathbf{x}$, we denote $\text{diag}(\mathbf{x})$ to be the diagonal matrix with diagonal entries contained in $\mathbf{x}$. For a matrix $\mathbf{X} \in \mathbb{R}^{d \times d}$, $\text{diag}(\mathbf{X}) = \text{diag}(\begin{bmatrix}
    \mathbf{X}(1,1), \mathbf{X}(2,2), \cdots, \mathbf{X}(d,d)
\end{bmatrix})$. Lastly, for a sequence of vectors $\mathbf{x}_1, \mathbf{x}_2, \cdots, \mathbf{x}_t$, $\text{diag}(\mathbf{x}_1, \mathbf{x}_2, \cdots, \mathbf{x}_t)$ denotes the diagonal block matrix consisting of $\text{diag}(\mathbf{x}_1), \text{diag}(\mathbf{x}_2), \cdots, \text{diag}(\mathbf{x}_t)$. For matrices, $\text{diag}(\mathbf{X}_1, \mathbf{X}_2, \cdots, \mathbf{X}_t)$ is the diagonal block matrix consisting of $\text{diag}(\mathbf{X}_1), \text{diag}(\mathbf{X}_2), \cdots, \text{diag}(\mathbf{X}_t)$.
\subsection{Graph Signal Processing Preliminaries and Terminologies}
 We first provide some formal definitions related to graph operators and signals. We consider a weighted and undirected graph $\mathcal{G} = (V, E, \mathbf{W})$, where $V = \{v_1, \cdots, v_n\}$ is a set of vertices, $E \subseteq V \times V$ is a set of edges, and $ \mathbf{W}$ is the weighted adjacency matrix. Specifically, if we denote $(v_i, v_j)$ as the edge between $v_i$ and $v_j$ if they are connected with positive weight, $w_{ij} > 0$. $ W$ is defined as
\[
	\mathbf{W}(i,j) =
\begin{cases}
	w_{ij}, & (v_i, v_j) \in E \\
	0, & \text{ otherwise}
\end{cases}.
\]  In particular, $ \mathbf{W}$ is symmetric because $\mathcal{G}$ is undirected.
The degree of a vertex $v_i$ is defined by $\text{deg}(v_i) = \sum_{j=1}^{n} \mathbf{W}(i,j)$, and the diagonal degree matrix of $\mathcal{G}$ is denoted by $\mathbf{D} = \text{diag}(\text{deg}(v_i))$.

With these definitions in place,  the Laplacian operator of $\mathcal{G}$ is defined as $ \mathbf{L} = \mathbf{D} - \mathbf{W}$.
$ \mathbf{L}$ is a positive-semidefinite operator, hence it admits an eigendecomposition $ \mathbf{L} = \mathbf{U} \mathbf{\Sigma} \mathbf{U}^{\top}$ where the columns of $U$ are orthonormal, and $\mathbf{\Sigma}$ is a diagonal matrix containing the eigenvalues $0 \leq \sigma_1 \leq \sigma_2 \leq \cdots \leq \sigma_n$. A classical result from spectral graph theory states that the multiplicity of the eigenvalue 0 corresponds to the number of connected components in the graph. To avoid technical complications, To avoid technical complications, we restrict our analysis to connected graphs, assuming each eigenvalue has multiplicity 1.  So we shall have $\sigma_1=0<\sigma_2<\cdots<\sigma_n$. 
A graph signal is a vector $ \mathbf{x} \in \mathbb{R}^{n}$ defined on the vertices $V$ of $\mathcal{G}$, i.e. $ \mathbf{x}(i)$ is the signal value associated with the node $v_i$. For any graph signal $ \mathbf{x}$ on $\mathcal{G}$, the {graph Fourier transform} of $ \mathbf{x}$ is defined as $ \hat{ \mathbf{x}} = \mathbf{U}^\top \mathbf{x}$, where $ \hat{ \mathbf{x}}$ contains the Fourier coefficients of $ \mathbf{x}$ ordered in increasing frequencies. The inverse Fourier transform is defined naturally as $ \mathbf{x} = \mathbf{U} \hat{ \mathbf{x}}$.

A graph signal can be considered smooth if neighboring nodes have similar signal values. As such, a common definition of smoothness for a graph signal $ \mathbf{x}$ to be the weighted sum of squared differences between every pair of neighboring nodes. This quantity can be represented via the Laplacian as \[ \mathbf{x}^{\top} {\mathbf{L}} \mathbf{x} = \sum_{\substack{(v_i, v_j) \in E \\ i \leq j}}w_{ij}( \mathbf{x}(i) - \mathbf{x}(j))^2.\] 

 For a smooth signal $ \mathbf{x}$, $ \mathbf{x}^{\top}  \mathbf{L} \mathbf{x}$ is small. Since $\mathbf{x}^\top \mathbf{L}\mathbf{x} = \sum_{i=1}^n \sigma_i \hat{\mathbf{x}}(i)^2$, we expect $ \hat{ \mathbf{x}}(i)$ to be close to zero for sufficiently large indices  $i$. This motivates the definition of $k$-bandlimited graph signals and its subset of $s$-sparse signals.
 
\begin{definition}
	A graph signal $ \mathbf{x}$ on $\mathcal{G}$ is $k$-bandlimited for some $k \in \mathbb{Z}_+$ if $ \mathbf{x} \in \text{span}( \mathbf{U}_k)$, where $ \mathbf{U}_k$ denotes the first $k$ columns of $ \mathbf{U}$. Equivalently, $ \mathbf{x}$ is bandlimited if the only nonzero entries of $ \hat{ \mathbf{x}}$ are in the first $k$ components. Furthermore, a $k$-bandlimited graph signal $ \bf{x}$ on $\mathcal{G}$ is $s$-sparse for $s \in \mathbb{Z}_+$ and $s \ll k$ if \[|\text{supp}( \hat{\bf{x}})| = |\{i : \hat{\bf{x}}(i) \neq 0 \}| \leq s.\] 
\end{definition}

This concept is   analogous to   sparse bandlimited signals in classical setting \cite{tropp2009beyond}, extended to finite graphs. In fact, we will show that in the best-case scenario,  the sampling complexity for our graph-based analog is comparable to that of the classical setting. In practice, when the graph signal is sufficiently smooth, we can choose a large upper bound $k$ and consider the signal to be additionally sparse within the subspace spanned by the first $k$ eigenvectors,  $\mathbf{U}_k$. 

In this paper, we will consider $s$-sparse $k$-bandlimited graph signals that evolve over time under a known linear operator $\mathbf{A}$, such as the heat diffusion operator. Our goal is to exploit the underlying physics to achieve efficient space-time sampling.

Now we restrict our attention to the heat diffusion process. Note that if $\mathbf{x} = \sum_{i=1}^k c_i \mathbf{u}_i$, then $\mathbf{A}^t \mathbf{x} = \sum_{i=1}^k c_i \lambda_i^t \mathbf{u}_i$. In other words, the diffusion process preserves the frequency domain. With this observation, we can introduce the following orthonormal space-time extension $\tilde{\mathbf{U}}_{k,T} \in \mathbb{R}^{Tn \times k}$ given by
\[
	\tilde{\mathbf{U}}_{k,T} = 
	\begin{bmatrix}
		\frac{1}{f_T(\lambda_1)} \mathbf{u}_1 & \frac{1}{f_T(\lambda_2)} \mathbf{u}_2 & \cdots & \frac{1}{f_T(\lambda_k)} \mathbf{u}_k \\
		\frac{\lambda_1}{f_T(\lambda_1)} \mathbf{u}_1 & \frac{\lambda_2}{f_T(\lambda_2)} \mathbf{u}_2 & \cdots & \frac{\lambda_k}{f_T(\lambda_k)} \mathbf{u}_k \\
		\vdots & \vdots & \cdots & \vdots \\
		\frac{\lambda_1^{T-1}}{f_T(\lambda_1)} \mathbf{u}_1 & \frac{\lambda_2^{T-1}}{f_T(\lambda_2)} \mathbf{u}_2 & \cdots & \frac{\lambda_k^{T-1}}{f_T(\lambda_k)} \mathbf{u}_k \\
	\end{bmatrix} \in \mathbb{R}^{Tn \times k},
    \] where $f_T(\lambda_i) = \sqrt{\sum_{j=0}^{T-1} \lambda_i^{2j}}$ is introduced to normalize the columns.

Define the space-time diffusive signal as \[
\pi_{\mathbf{A},T}(\mathbf{x}) = \begin{bmatrix}
    \mathbf{x}^\top & (\mathbf{A} \mathbf{x})^\top & \cdots & \left( \mathbf{A}^{T-1}\mathbf{x}\right)^\top
\end{bmatrix}^\top,
\] for which we observe that for any $s$-sparse $k$-bandlimited $\mathbf{x}$,  $\tilde{U}_{k,T}^\top \pi_{\mathbf{A},T}(\mathbf{x})$ remains $s$-sparse. Indeed, if $\mathbf{x}$ is bandlimited of the form $\mathbf{x} = \sum_{i=1}^k   c_i \mathbf{u}_i$, denote $\mathbf{c} = \left[ c_1, \cdots, c_k \right]^\top$, then 
 $\pi_{\mathbf{A},T}(\mathbf{x}) =  \tilde{\mathbf{U}}_{k,T}\text{diag}\left([f_T(\lambda_i)]_{i=1}^k\right)\mathbf{c}$ with $\text{diag}([f_T(\lambda_i)]_{i=1}^k)\mathbf{c}$ being $s$-sparse. Hence we have transformed to dynamic signal into a static one with respect to the dictionary $\tilde{\mathbf{U}}_{k,T}$. 
We introduce the following inequality from \cite{dynamic_sampling}, which will be useful in future discussions. 
\begin{prop}[Lemma 4.3 in \cite{huang2020reconstruction}]\label{prop:embed}
    For any $\mathbf{x} \in \text{span}(\mathbf{U}_k)$, \[
    f_T(\lambda_k) \|\mathbf{x}\|_2 \leq \|\pi_{\mathbf{A},T}(\mathbf{x})\|_2 \leq f_T(\lambda_1) \|\mathbf{x}\|_2.
    \]
\end{prop}

\section{Dynamical Sampling and Graph Signal Recovery}
Recall that  $\mathbf{x} = \mathbf{x}_0$ is  $s$-sparse and $k$-bandlimited, $\mathbf{A}$ is a heat diffusion operator, and $\mathbf{x}_t = \mathbf{A}^{t} \mathbf{x}_0$ for $t \in \{0\}\cup[T-1]$. We assume that samples are taken from $\{\mathbf{x}_t\}_{t=0}^{T-1}$ at each time step according to some probability distribution $\mathbf{p}_t$. Our goal is to establish precise conditions on the number of samples required for robustly recovering the initial signal $\mathbf{x}$.

\paragraph{Formulating the space-time sampling problem as an $\ell_1$ recovery problem}

Note that our sampling and reconstruction problem can be formulated as a linear inverse problem 
\begin{align}\label{orp}
\mathbf{y} = \mathbf{S}(\pi_{\mathbf{A},T}(\mathbf{x})) + \mathbf{e}
\end{align}
where \(\mathbf{e}\) is some observational noise. We aim to exploit the sparsity of \(\mathbf{x}\) in the spectral domain. To do this, we introduce a set of operators to transform our problem into the spectral domain.

\begin{definition}\label{def:probs}
\begin{enumerate}[label=(\roman*),leftmargin=.25in]
    \item \textit{The Sampling selection operator}: Define the sampling set by \(\Omega= \bigcup_{t=0}^{T-1} \Omega_t\) where \(\Omega_t=\left\{\omega_{t,1}, \cdots, \omega_{t,m_t}\right\}\) is constructed by drawing \(m_t\) indices independently (with replacement) from \([n]\) according to \(\mathbf{p}_t\), i.e., \(\mathbb{P}(\omega_{t,j} = i)=\mathbf{p}_t(i)\), for all \(j \in[m_t]\) and \(i \in[n]\). We define the sampling matrix
    \[
    \mathbf{S}=\operatorname{diag}(\mathbf{S}_0 ; \cdots ; \mathbf{S}_{T-1}),
    \]
    where \(\mathbf{S}_t \in \mathbb{R}^{m_t \times n}\) is defined as 
    \[
    \mathbf{S}_t(i, j)=\begin{cases}
    1, & \text{ if } j=\omega_i \\ 
    0, & \text{ otherwise }
   \end{cases}.
    \]
    The total number of space-times samples  is \(M = \sum_{t=0}^{T-1} m_t\).
    
    \item \textit{The reweighting operator}: Let \(\mathbf{p}_t \in \mathbb{R}^n\) be probability distributions on the nodes \([n]\) of \(\mathcal{G}\) for \(t \in \{0\}\cup[T-1]\). We define the probability matrix \(\mathbf{P} \in \mathbb{R}^{Tn \times Tn}\) by \(\mathbf{P} = \text{diag}(\mathbf{p}_0, \cdots, \mathbf{p}_{T-1})\) and we introduce the matrices \(\mathbf{P}_{\Omega} = \mathbf{S} \mathbf{P} \mathbf{S}^\top\) and 
    \[
    \mathbf{W} = \text{diag}\left( \frac{1}{m_0} I_{m_0}, \cdots, \frac{1}{m_{T-1}} I_{m_{T-1}}\right).
    \]
    The reweighted operator is defined as \(\mathbf{\Psi} = \mathbf{P}_{\Omega}^{-\frac{1}{2}} \mathbf{W}^{\frac{1}{2}}\). Here we assume the probability distributions are nondegenerate, i.e., all entries are nonzero. The associated probability matrices \(\mathbf{P}\) are not needed in the node selection process, but will return as a part of a reweighting matrix in the recovery process.
\end{enumerate}
\end{definition}

\begin{prop}
Assume that \(\mathbf{x}\) is precisely \(s\)-sparse, and \(\mathbf{e}=0\). Then \eqref{orp} can be equivalently formulated as 
\begin{equation}
\mathbf{\Psi} \mathbf{S} \tilde{\mathbf{U}}_{k,T}\mathbf{c} =  \widetilde{\mathbf{y}}
\end{equation}
where we define \(\widetilde{\mathbf{y}}=\mathbf{\Psi}\mathbf{y}\). 
\end{prop}

Here the re-weighted operator $\mathbf{\Psi}$ is applied on \eqref{orp} to ensure numerical stability and later we show the re-weighted sampling matrix is well conditioned with high probability. Under this formulation, the random matrix \(\mathbf{\Psi} \mathbf{S}\tilde{\mathbf{U}}_{k,T}\) is a sampling operator on sparse vectors. To exploit sparsity, \(\ell_1\)-minimization is natural to consider. Therefore, we propose 

\begin{equation}
\label{eq:noiseless_problem}
\min_{\mathbf{c} \in \mathbb{R}^k} \|\mathbf{c}\|_1 \text{ subject to } \underbrace{\mathbf{\Psi} \mathbf{S} \tilde{\mathbf{U}}_{k,T}}_{\Phi}\mathbf{c} =  \widetilde{\mathbf{y}}.
\end{equation}

Whenever \(\mathbf{e} \neq 0\), we propose to recover \(\mathbf{x}\) from the samples \(\mathbf{y}\) by solving the following minimization problem:
\begin{equation}
\label{eq:noisy_problem}
\min_{\mathbf{c} \in \mathbb{R}^k} \|\mathbf{c}\|_1 \text{ subject to } \left\|\mathbf{\Psi}\mathbf{S} \tilde{\mathbf{U}}_{k,T}\mathbf{c} -  \widetilde{\mathbf{y}}\right\|_2 \leq \eta,
\end{equation}
with \(\eta\) chosen such that \(\eta \geq \|\mathbf{\Psi} \mathbf{e}\|_2\) to incorporate the noise level.

There are many algorithms available to solve the \(\ell_1\)-minimization problem. Some examples include LASSO \cite{tibshirani_lasso}, OMP \cite{TroppOMP}, or IHT \cite{blumensath2009iterative}. Most \(\ell_1\)-minimization algorithms work well for our recovery tasks, but we choose the CoSaMP algorithm \cite{cosamp} in all experiments for its simple set of parameters and implementation. A pseudocode implementation \cite{Foucart2013-sv} of the CoSaMP algorithm is given in \Cref{alg:cosamp}. \(H_s\) denotes the hard-thresholding operator of order \(s\), where for any vector \(\mathbf{c}\), \(H_s\) keeps the \(s\) largest entries in absolute value (where tiebreaks are determined by lexicographic order) and sets the other entries to \(0\).

Note that the recovery algorithm is applicable as long as \(\mathbf{x}\) is \(k\)-bandlimited, rather than also exactly \(s\)-sparse. Sufficiently smooth \(\mathbf{x}\) are \textit{approximately} sparse in the frequency domain, and algorithms such as CoSaMP allow for specifying a target sparsity level as a parameter. The output of CoSaMP then provides the best \(s\)-sparse solution, and thereby produces the best \(s\)-sparse approximation to \(\mathbf{x}\).

Our main objective moving forward is to minimize the resources required so that the frequency vectors produced by the recovery Problems \eqref{eq:noiseless_problem} and \eqref{eq:noisy_problem}  are close to the true solution. There are two aspects to this objective. One aspect is the sampling cost.  We aim to minimize the number of samples required to guarantee accurate recovery. For a particular \(s\)-sparse \(k\)-bandlimited signal, it is known that there exists a deterministic sampling set of size \(2s\) that uniquely recovers the signal. In our case, we wish to design the probability distributions in \Cref{def:probs} so that \textit{all} approximately \(s\)-sparse and \(k\)-bandlimited signals can be recovered robustly with as few samples as possible. Another aspect is the computational cost. This cost is algorithm-dependent and can vary on a case-by-case basis due to the random sampling procedure. We will discuss when our random sampling procedure generates a sampling operator \(\Phi\) that satisfies a restricted isometry property, and we will provide the CoSaMP-specific computational costs.

\subsection{Restricted Isometry Property of  $\mathbf{\Psi} \mathbf{S} \tilde{\mathbf{U}}_{k,T}^\top$}
In compressed sensing, the Restricted Isometry Property (RIP) plays a pivotal role in ensuring the uniqueness and stability of solutions in recovery problems. Recall that a matrix $ \mathbf\Phi $ meets the $ s $-RIP criteria with a constant $ \delta_s \in (0,1) $ if, for every $ s $-sparse vector $ \mathbf{z} $, the following relationship is maintained:

\[
(1 - \delta_s) \| \mathbf{z}\|_2^2 \leq \|\mathbf\Phi \mathbf{z}\|_2^2 \leq (1 + \delta_s) \|\mathbf{z}\|_2^2.
\]

In literature, the RIP properties for random matrices with i.i.d. entries sampled from well-behaved probability distributions are well established. In the context of our recovery problem \eqref{eq:noisy_problem}, our focus lies in identifying the required sample size to ensure the RIP for the sampling operator \(\mathbf{\Psi} \mathbf{S}\tilde{\mathbf{U}}_{k,T}^\top\), which are structured and constrained by dynamical systems. This prevents us from directly using the existing results and necessitates nontrivial efforts, which we shall discuss in \Cref{thm:RIP}.

We first introduce a graph dynamical weighted incoherence parameters that characterize the sampling complexity required to ensure the RIP   of the corresponding sampling operator.  

\begin{definition}[Dynamic spectral graph weighted coherence for sparse recovery]\label{def:coherence}
Let $\mathcal{S} \subset [k]$ be a specific subset containing $s$ elements and we call the collections of all the subset of $[k]$ with $s$ elements $s$-subsets. The parameter $\nu_{k,T}$ represents the dynamical spectral graph weighted coherence parameter and its definition is based on the restrictions of the matrix $\widetilde{\mathbf{U}}_{k,T}$ on various $s$-subsets. To ease of the presentation, we shall drop the indices $k,T$, and refer to the submatrix of $\widetilde{\mathbf{U}}_{k,T}$ with column indexed by $\mathcal{S}$ and rows from  $t n + 1$ to $(t+1)n$ as $\tilde{\mathbf{U}}_{\tau_t,\mathcal{S}}$ i.e., $\tilde{\mathbf{U}}_{\tau_t,\mathcal{S}}:=\widetilde{\mathbf{U}}_{k,T}(tn+1:(t+1)n,\mathcal{S})$.  At time $t = 0, 1, \cdots, T-1$, define \[\nu_{k,t} = \max\limits_{\mathcal{S} \subset [k]} \max\limits_{1 \leq i \leq n} \frac{\|\tilde{\mathbf{U}}_{\tau_t,\mathcal{S}}^\top \delta_i\|_\infty}{\sqrt{\mathbf{p}_t(i)}}=\max\limits_{\{j\} \subset [k]} \max\limits_{1 \leq i \leq n} \frac{\|\tilde{\mathbf{U}}_{\tau_t,j}^\top \delta_i\|_\infty}{\sqrt{\mathbf{p}_t(i)}}= \max_{i\in [n],j\in [k]} \frac{|\tilde{\mathbf{U}}_{tn+i,j}|}{\sqrt{\mathbf{p}_t(i)}},\] where the last two identities follows from the property of $\ell^{\infty}$ norm. We then define $\nu_{k,T} = \begin{bmatrix}
            \nu_{k,0}, \cdots, \nu_{k,T-1}
        \end{bmatrix}^\top$, so that $\nu_{k,T}(t) = \nu_{k,t}$. 
       \end{definition}

The coherence parameter introduced in \Cref{def:coherence} differs from that used in the bandlimited case without a sparsity constraint \cite{huang2020reconstruction}. Specifically, we use the \(\ell_{\infty}\) norm for the rows of \(\tilde{\mathbf{U}}\), whereas the bandlimited case relies on the \(\ell_2\) norm. 
It is also worth noting that, when the eigenvalues of the Laplacian matrix have multiplicities, its eigen-decomposition is not unique. Our approach suggests that selecting an eigen-basis with the most spread out vectors, minimizing the infinity norm for each column, is desirable.

 Our definition is independent of the signal's sparsity, which may not fully capture variations in sparsity. However, we show that the sampling complexity for our recovery problem at each time step is quadratically proportional to this parameter, ensuring the RIP  of the sampling operator, summarized below.

\begin{theorem}[RIP   of sampling operator]\label{mainthm} Let $\delta\in (0,1)$, if 
\begin{align}\label{ripbound3}
m_t \geq C\nu_{k,t}^2\delta^{-2}s\log^2(s)\log(k)\log(nT), t=0,\cdots, T-1,
\end{align}then with probability at least $1-(nT)^{-\log^2(s)\log(k)}$ the restricted isometry constant $\delta_s$ of the sampling operator $\mathbf{\Psi} \mathbf{S} \tilde{\mathbf{U}}_{k,T}$ satisfies $\delta_s \leq 3\delta$. The constant $C>0$ is universal. 
\end{theorem}
\begin{proof}First, it follows by a ratio test that \eqref{ripbound3} with an appropriate universal constant $C>0$ implies \eqref{ripbound2} with $\eta=\delta$. Second, if we choose $\epsilon =(nT)^{-\log^2(s)\log(k)}$, then \eqref{ripbound3} implies \eqref{ripbound1} in  \Cref{thm:RIP} with $\beta=\delta$. Therefore,  \Cref{thm:RIP} implies $\delta_s\leq 3\delta$ with probability at least $1-(nT)^{-\log^2(s)\log(k)}$. 
\end{proof}

In Section \ref{sec:min_complexity} below, we discussed how to minimize the coherence. We show in some scenarios, $\nu_{k,t}^2$ can be as small as $\frac{1}{T}$, in this case, the sampling abound of $m_t$ is nearly proptional to classical case by a factor $1/T$, this shows a perfect time-space trade off. More numerical results will be presented in Section \ref{sec:experiments}.

\subsection{Recovery Guarantees}\label{sec:analysis}
In general, the sampling complexity for unique or stable recovery is algorithm-dependent. Since we are using the CoSaMP algorithm, we will provide the recovery theorem for CoSaMP using the notation introduced in \Cref{alg:cosamp}. Additionally, we denote $\mathcal{S} \subset [k]$ for subsets with $s$ elements, $\mathbf{c}_\mathcal{S}$ as the vector $\mathbf{c}$ with entries outside $\mathcal{S}$ set to zero, and $\mathbf{c}_{\overline{\mathcal{S}}}$ as the vector $\mathbf{c}$ with entries on $\mathcal{S}$ to zero.

We present the recovery theorem for the noisy setting \eqref{eq:noisy_problem}, where $\mathbf{c} = \tilde{\mathbf{U}}_{k,T}^\top \mathbf{x}$ is only approximately $s$-sparse, and our observation is contaminated with additive noise. 
The $\ell_1$-error of best $s$-sparse approximation to a vector $\mathbf{c}$ is defined by 
\[
\sigma_s(\mathbf{c}) = \inf_{\|\mathbf{z}\|_0 \leq s} \|\mathbf{c} - \mathbf{z}\|_1. 
\]
\begin{prop}\label{thm:cosamp_stable}
    Let $\mathbf{y} = \mathbf{S} \pi_{\mathbf{A},T}(\mathbf{x}) + \mathbf{e}$ be a noisy observation of a $k$-bandlimited graph signal $\mathbf{x}$ that is not necessarily $s$-sparse, and denote $\mathbf{c} = \tilde{\mathbf{U}}_{k,T}^\top \pi_{\mathbf{A},T}(\mathbf{x})$. Suppose that $m_t \geq \tilde{C}\nu_{k,t}^2 s\log^2(8s)\log(k)\log(nT), t=0,\cdots, T-1$,   then the $n$th iteration $\mathbf{c}^n$ generated by the CoSaMP  \Cref{alg:cosamp} satisfies the following bounds with probability at least $1-(nT)^{-\log^2(8s)\log(k)}$, we have \begin{align*}
        \|\mathbf{c} - \mathbf{c}^n\|_2 &\leq \frac{C_1}{\sqrt{s}} \sigma_s(\mathbf{c}) + D \|\Psi \mathbf{e}\|_2 + 2 \rho^n \|\mathbf{c}\|_2,
    \end{align*}
    where the constant $\tilde{C} > 0$ is universal, and the constants $C_1, D > 0$ and $0 < \rho < 1$ depend only on $\delta_{8s}$. 
\end{prop}

\begin{proof}[Proof of \Cref{thm:cosamp_stable}]

According to Theorem 6.28 from \cite{Foucart2013-sv}, it suffices to require $\delta_{8s} < 0.4782$, where $\delta_{8s}$ denotes the $8s$-restricted isometry constant of $\mathbf{\Psi} \mathbf{S} \tilde{\mathbf{U}}_{k,T}$. By  Theorem \ref{mainthm}, $\delta_{8s} < 0.4782$ can be guaranteed by replacing $s$ with $8s$, and taking $\delta = 0.15$ in \eqref{ripbound3}, which yields the results in our proposition. 
      
\end{proof}

\begin{remark}
    Note that since $\tilde{\mathbf{U}}_{k,T}$ contains orthonormal columns, we have $\|\mathbf{c}\|_2 = \|\tilde{\mathbf{U}}_{k,T} \mathbf{c}\|_2 = \|\pi_{\mathbf{A},T}(\mathbf{x})\|_2$. We denote $\tilde{\mathbf{U}}_{k,T}\mathbf{c}^n = \pi_{\mathbf{A},T}(\mathbf{x}^n)$, then   \Cref{thm:cosamp_stable} together with   \Cref{prop:embed} implies 
\begin{align*}
        f_T(\lambda_k)\|\mathbf{x} - \mathbf{x}^n\|_2 &\leq \frac{C_1}{\sqrt{s}} \sigma_s(\mathbf{c}) + D \|\mathbf{\Psi} \mathbf{e}\|_2 + 2 \rho^n f_T(\lambda_1)\|\mathbf{x}\|_2.
    \end{align*} 
    In particular, if $\mathbf{c}^\sharp$ is an accumulation point of $\mathbf{c}^n$, denote $\tilde{\mathbf{U}}_{k,T} \mathbf{c}^\sharp = \pi_{\mathbf{A},T}(\mathbf{x}^\sharp)$, then $\mathbf{x}^\sharp$ satisfies 
    \begin{align*}
        f_T(\lambda_k)\|\mathbf{x} - \mathbf{x}^\sharp\|_2 &\leq \frac{C_1}{\sqrt{s}}\sigma_s(\mathbf{c}) + D \|\mathbf{\Psi} \mathbf{e}\|_2.
    \end{align*}
\end{remark} In other words, when $\mathbf{y}$ is noisy and $\mathbf{c}$ is only approximately sparse, the recovery error on the initial signal $\mathbf{x}$ is bounded by a linear combination of the noise term and how much $\mathbf{c}$ differs from an $s$-sparse vector. Further if $\mathbf{e} = \mathbf{0}$ and $\mathbf{x}$ is exactly $s$-sparse, $\mathbf{x}^n \rightarrow \mathbf{x}$, which guarantees the uniqueness of recovery.

\subsection{Optimizing the sampling probability distribution}\label{sec:min_complexity}
In applications, we would like to minimize the sampling complexity to guarantee a RIP according to \Cref{thm:RIP}. Since we cannot control the bandwidth $k$ nor the sparsity level $s$ of real signals, we are left with minimizing $\nu_{k,T}$. The following proposition defines a set of optimal sampling distributions that minimize their respective coherence parameters  
\[\nu_{k,T}(t) = \max_{i\in [n],j\in [k]} \frac{|\tilde{\mathbf{U}}_{tn+i,j}|}{\sqrt{\mathbf{p}_t(i)}}\].

\begin{prop}\label{prop:optimal}
The following choice of probability distributions minimize their respective coherence parameters.  For $t = 0,1, \cdots, T-1$, $\nu_{k,T}(t)$ is minimized when \begin{align} \label{optprob}
\mathbf{p}_t(i) = \frac{|\tilde{\mathbf{U}}_{tn+i,\kappa_i}|^2}{\sum_{j=1}^n |\tilde{\mathbf{U}}_{tn+j,\kappa_j}|^2},\end{align}
where $\kappa_i = \arg\max\limits_{m \in [k]} |\tilde{\mathbf{U}}_{tn+i,m}|$ for $i \in[n]$. In this case, 

\begin{align}\label{optcoherence}
        \left(\nu_{k,T}(t)\right)^2 = \sum_{j=1}^n |\tilde{\mathbf{U}}_{tn+j,\kappa_j}|^2
\end{align}
       
\end{prop}

\begin{proof}[Proof of   \Cref{prop:optimal}]\label{prf:5}
  The formula of \eqref{optcoherence} follows from direct calculations. Now we prove the optimality of \eqref{optprob}. 
   Suppose $\mathbf{q}$ is a probability distribution not equal to $\mathbf{p}_t$, then there exists some index $\ell$ for which $\mathbf{q}(\ell) < \mathbf{p}_t(\ell)$. As a result, 
    \begin{align*}
        \max\limits_{i \in [n],j\in [k]}  \frac{|\tilde{\mathbf{U}}_{tn+i,j}|^2}{\mathbf{q}(i)} &\geq \max\limits_{j \in [k]} \frac{|\tilde{\mathbf{U}}_{tn+\ell,j}|^2}{\mathbf{q}(\ell)} 
        > \max\limits_{j\in [k]} \frac{|\tilde{\mathbf{U}}_{tn+\ell,j}|^2}{\mathbf{p}_t(\ell)} 
       =\frac{|\tilde{\mathbf{U}}_{tn+\ell,\kappa_{\ell}}|^2}{\mathbf{p}_t(\ell)} =  \sum_{j=1}^n |\tilde{\mathbf{U}}_{tn+j,\kappa_j}|^2.
    \end{align*}
\end{proof}

We will henceforth refer to the probability distributions introduced in \Cref{prop:optimal} as optimal distributions. Next we analyze this incoherence parameter for some graphs. 

\begin{itemize}
\item \textbf{Circulant graphs with periodic boundary conditions}. In this case, their adjacent matrices are real and symmetric circulant matrices.  The eigenvalues will have multiplicities but we choose  the eigenvectors of their Laplacian to be the $n$ dimensional classical Fourier modes. For simplicity, assume that the top $k$ eigenvalues are $1=\lambda_1\geq \lambda_2\geq \cdots,\lambda_k>0$.  Then,   with fixed time instance $t$ and fixed column index $j$, $| \widetilde{\mathbf{U}}_{tn+i,j}|$ is the same for all node $i$. Therefore $\kappa_i$ is the same for all $i$, and
$\kappa_i \equiv \mathrm{argmax}_{j\in k} \frac{\lambda_j^{2t}}{f_{T}^2(\lambda_j)}=j_0$. This indicates that the optimal distribution $\mathbf{p}_t$ is the uniform distribution {for every $t=0,\cdots,T-1$}.  In fact, we have 
\[
\left(\nu_{k,T}(t)\right)^2 = 
\begin{cases} 
\sum_{j=1}^n |\tilde{\mathbf{U}}_{tn+j,j_0}|^2 = \frac{\lambda_{j_0}^{2t}(1-\lambda_{j_0}^2)}{1-\lambda_{j_0}^{2T}} & \text{if } \lambda_{j_0} \neq 1, \\
\sum_{j=1}^n |\tilde{\mathbf{U}}_{tn+j,j_0}|^2  = \frac{1}{T} & \text{if } \lambda_{j_0} = 1.
\end{cases}
\]

        so that 
        $1 \leq \sum_{t=0}^{T-1}\nu_{k,T}^2(t) < T $. Note that if $t=0$, then {$\nu_{k,T}^2(t)= \frac{1}{f_T^2(\lambda_{j_0})}$}.  We note that in slow diffusion regime where all $\lambda_i$ are close to 1, one is expected to get $\nu_{k,T}^2(t)\approx \frac{1}{T}$, so $\sum_{t=0}^{T-1}\nu_{k,T}^2(t) \approx 1$. In fast diffusion regime, we have only a few eigenvalues equal or close to 1, and all other eigenvalues close to 0, in this case,  one can also get $\nu_{k,T}^2(t)\approx \frac{1}{T}$. For intermediate regimes, we do not have closed formulas and we will explore this in the numerical section. One can refer the results in  Figure \ref{tab:coh_r} and \ref{tab:coh_m}, where we see that in all cases, $\sum_{t=0}^{T-1}\nu_{k,T}^2(t)$ are $o(1)$ and do not scale with $T$. 

  \item \textbf{The community graph made of $k$ disconnected components of size $n_1,\cdots, n_k$, denoted by $\mathcal{C}_1,\cdots,\mathcal{C}_k$, such that $\sum_{i=1}^{k}n_i=n$}.  Let the Laplacian $\mathbf{L}$ be the  graph Laplacian with all edge weights  equal to one. Then a basis of $\mathrm{span}(\mathbf{U}_k)$ is the concatenation of the square root of the degree operator multiplied by the indicator vectors of each connected component. Moreover, $\mathrm{span}(\mathbf{U}_k)$  is the eigenspace associated with the eigenvalue 0 of $\mathbf{L}$. As a result, the heat diffusion process over this eigenspace is the trivial identity. So the temporal dynamics plays no role in sampling and it will be the same with the static case.  
  
  By calculation, for  node $i$ in component $j$, we have $\|\widetilde{\mathbf{U}}_{tn+i,j'} \|_{\infty} = 0$ for $j'\neq j$ and $\|\widetilde{\mathbf{U}}_{tn+i,j} \|_{\infty} $ is the same for all $t$ with
\begin{equation*} 
\|\widetilde{\mathbf{U}}_{tn+i,j} \|_{\infty} =\frac{\sqrt{d_{ij}}}{\sqrt{d_jT}}, 
\end{equation*}where $d_{ij}$ is the degree of  node $i$ in component $\mathcal{C}_j$ and $d_j$ is the total degree of nodes in component $j$.  For simplicity, let us assume that  each component is regular, i.e., each node has the same degree. In this case,  for $i\in \mathcal{C}_j$, $\kappa_i=j$ and therefore the optimal probability of choosing the spatial node $i$  in $\mathcal{C}_j$ at time $t$ is $\frac{1}{kn_j}$. Thus, if all components have the same size, then uniform sampling is the optimal sampling. However, if components differ in size, smaller components should have a higher sampling probability.   We can see that optimal sampling requires that each component is sampled with probability $\frac{1}{k}$, independent of its size. The probability that each component is sampled at least once, a necessary condition for perfect recovery, is thus higher than that using uniform sampling. In the numerical section (Section \ref{sec:experiments}), we consider the loosely defined
community-structured graph, which is a  relaxation of the strictly  disconnected component example. In this case, one also expects that the optimal $\mathbf{p}_t$ to sample a node is inversely proportional to the size of its community.
\end{itemize}


\section{Extensions}\label{sec:generalize}

We have primarily focused on the linear heat diffusion process as a linear dynamical system with the core assumption that the graph signals are bandlimited and sparse in the graph frequency domain. Our sampling theory framework is also applicable more generally.

\begin{enumerate}[label=(\roman*),leftmargin=.25in]
    \item \textbf{More General Evolution Operators:}
    \begin{enumerate}[label=(\alph*)]
        \item \textbf{Functions of Laplacian operators:} The evolution operator can be a continuous function of the graph Laplacian: Let $h$ be a continuous function defined on a compact interval containing the spectrum of the graph Laplacian, $\mathbf{A}=h(\mathbf{L})$. Such operators are diagonalizable and shares the same eigenvectors with the graph Laplacian. This setup enables frequency-preserving dynamics, where the graph signal evolves within the eigenspace spanned by the top $k$ eigenvectors $\{\mathbf{u}_i\}_{i=1}^{k}$ of the graph Laplacian $\mathbf{L}$.

        \item \textbf{Orthogonality-Preserving Operators:} We can also consider operators $\mathbf{A}$ that preserve orthogonality relations between $\{\mathbf{u}_i\}_{i=1}^k$ over time. For example, if $\mathbf{A} \in \{c \cdot \mathbf{M} : \mathbf{M} \in O(n)\}$, where $O(n)$ denotes the orthogonal group, then $\tilde{\mathbf{U}}_{k,T}$ can be constructed as a column-wise normalization of the matrix:
        \[
        \begin{bmatrix}
            \mathbf{u}_1 & \mathbf{u}_2 & \cdots & \mathbf{u}_k \\
            \mathbf{A} \mathbf{u}_1 & \mathbf{A} \mathbf{u}_2 & \cdots & \mathbf{A} \mathbf{u}_k \\
            \vdots & \vdots & & \vdots \\
            \mathbf{A}^{T-1} \mathbf{u}_1 & \mathbf{A}^{T-1} \mathbf{u}_2 & \cdots & \mathbf{A}^{T-1} \mathbf{u}_k
        \end{bmatrix}.
        \]
    \end{enumerate}
 \item \textbf{Irregular observation time instances.} In our work, we assumed regular observation time instances. However, our framework can be extended to handle sampling at arbitrary time instances \( t_0 < t_1 < \cdots < t_l = T \), leveraging the fact that the diffusion matrices at different time points can be diagonalizable simultaneously.

    \item \textbf{From Linear to Affine Linear Dynamical Systems:} Consider discrete affine dynamical systems governed by $\mathbf{x}_{t+1} = \mathbf{A} \mathbf{x}_t + \mathbf{b}$. In this setting, we construct $\tilde{\mathbf{U}}_{k,T}$ from $A$ alone, disregarding the $\mathbf{b}$ term. If $\mathbf{b}$ is known, during the sampling process, we can manually eliminate the influence of $\mathbf{b}$, represented by:
    \[
    \begin{bmatrix}
        0 & \mathbf{b} & \mathbf{A} \mathbf{b} + \mathbf{b} & \cdots & \mathbf{A}^{T-2} \mathbf{b} + \cdots + \mathbf{b}
    \end{bmatrix}^\top.
    \]
    Alternatively, if $\mathbf{I} - \mathbf{A}$ is invertible, we can apply a coordinate shift $\mathbf{x} \mapsto \mathbf{x} + (\mathbf{I} -\mathbf{ A})^{-1}\mathbf{b}$, so the shifted signals evolve solely based on powers of $\mathbf{A}$.
\end{enumerate}

\section{Experiments}\label{sec:experiments}

In this section, we empirically investigate two aspects of the recovery problem: (i) verifying the sampling complexity provided by our sampling theorems in both noise-free and noisy regime; (ii) studying the behavior of dynamical graph coherence for ring and Minnesota graphs. The structure of this section is organized as follows:

\begin{enumerate}[label=(\roman*),leftmargin=.25in]
    \item In \Cref{sec:exp_coh}, we examine the behavior of $\left(\nu_{k, T}\right)^2$ across different regimes and sampling distributions applied to both the ring graph and the Minnesota graph (refer to \Cref{fig:graphs}).
    \item In \Cref{sec:noiseless}, we confirm that the sampling complexity required for recovery in noiseless settings is proportional to $s$, using synthetic data.
    \item In \Cref{sec:noisyrecovery}, we test the robustness of algorithms described in \Cref{thm:cosamp_stable}. 
\end{enumerate}

  We use the CoSaMP algorithm \cite{cosamp} with 20 iterations for all recovery tests. For the experiments in  \Cref{sec:exp_coh,sec:noiseless}, we consider the Minnesota graph and the unweighted ring graph each with $n = 2642$ from the GSP Toolbox \cite{perraudin2014gspbox}. A visualization of the two graphs is provided in \Cref{fig:graphs}.

\begin{figure}[ht]
    \centering
    \includegraphics[width=0.4\linewidth]{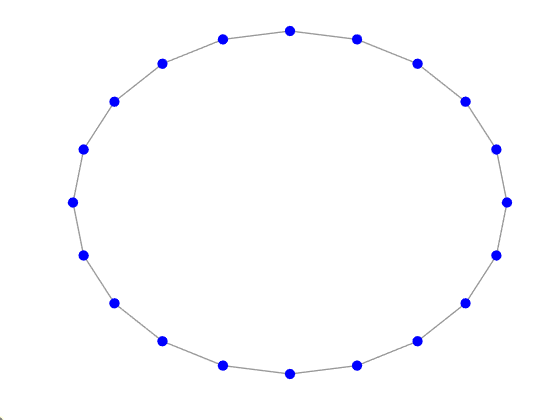}
    \includegraphics[width=0.4\linewidth]{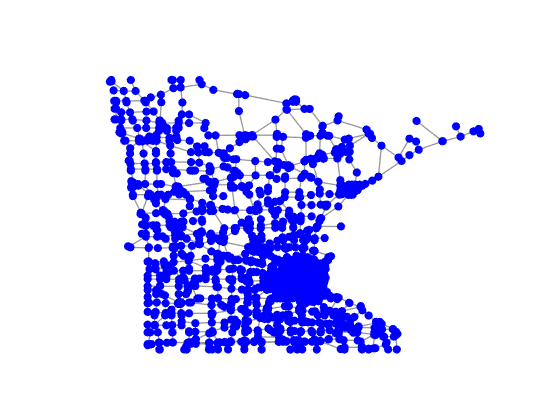}
    \caption{\footnotesize A downsized ring graph (left) and the Minnesota graph (right).}
    \label{fig:graphs}
\end{figure}

\subsection{Coherence Parameter $\left(\nu_{k,T}\right)^2$}\label{sec:exp_coh}
For \( t \in \{0\}\cup[T-1] \), we wish to study the value \( \left(\nu_{k, T}(t)\right)^2 \) as a function of \( k \) and \( T \) for the optimal and uniform sampling distributions on the ring and Minnesota graphs, as shown in \Cref{fig:graphs}. These two graphs represent instances of regular and irregular graphs, respectively. The value \(\nu_{k, T}(t)\) represents the spatio-temporal trade-off ratio at a single time \( t \). The sum \(\sum_{t=0}^{T-1} \left(\nu_{k, T}(t)\right)^2\) represents the spatio-temporal trade-off ratio over all time instances, as the former characterizes the sample complexity needed at time \( t \), while the latter characterizes the total sample complexity needed over all time instances.

On both graphs,  we consider $T = 1:1:20$ and $k= 100:100:1000$ to compute the uniform and optimal coherence values. Particularly, the results for the case of  $T=10$ and $k=1000$ are shown in \Cref{tab:coh_r}. The general results are shown in  \Cref{tab:coh_m}.

\begin{figure}
    \centering
    \includegraphics[width=0.8\linewidth]{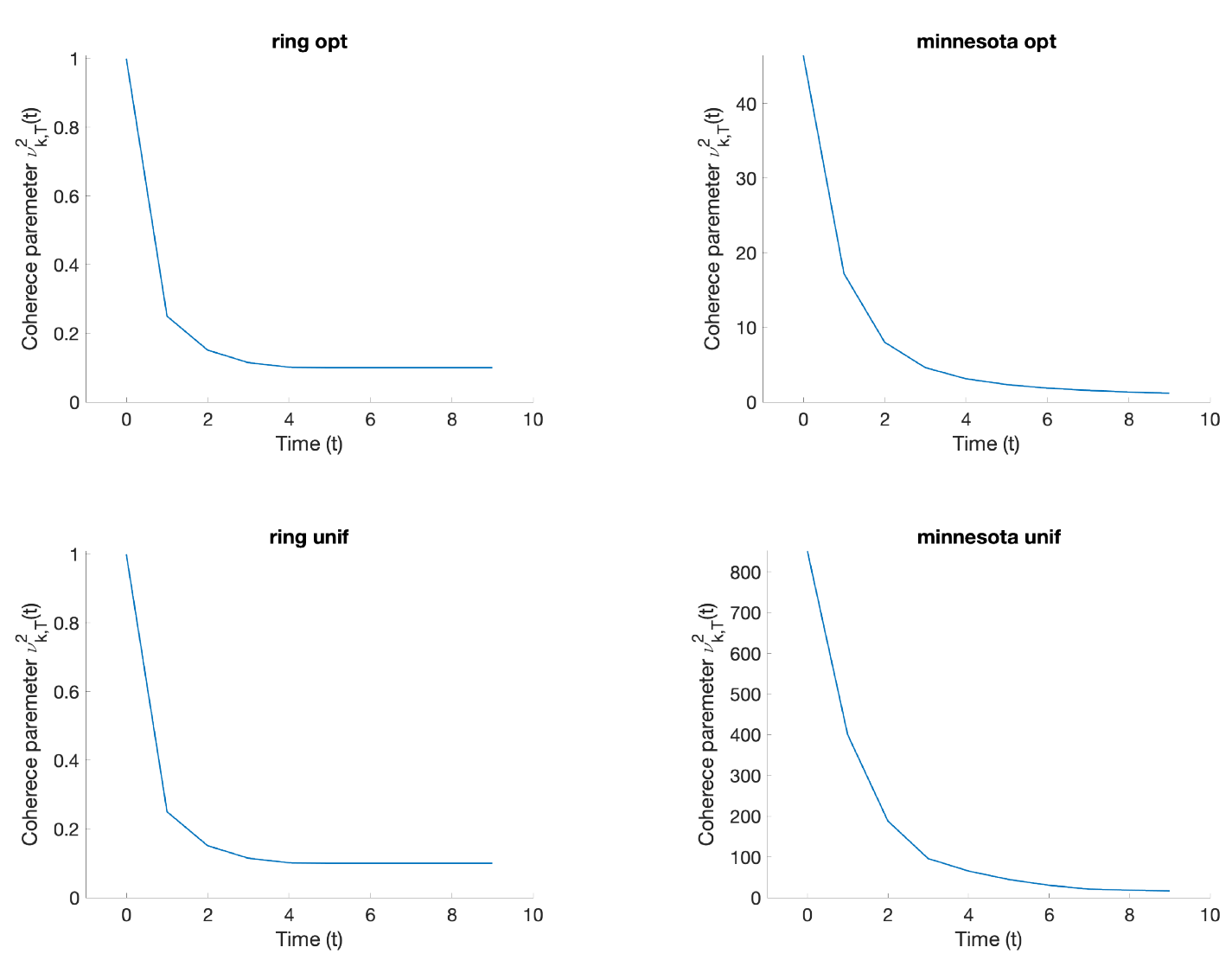}
    \caption{$\nu^2_{k,T}(t)$ plot with fixed $T=10$  and $k=1000$  for ring graph (left) and Minnesota graph (right). Top is for optimal sampling and bottom is for uniform sampling.}
    \label{tab:coh_r}
\end{figure}

\begin{figure}
    \centering
    \includegraphics[width=\linewidth]{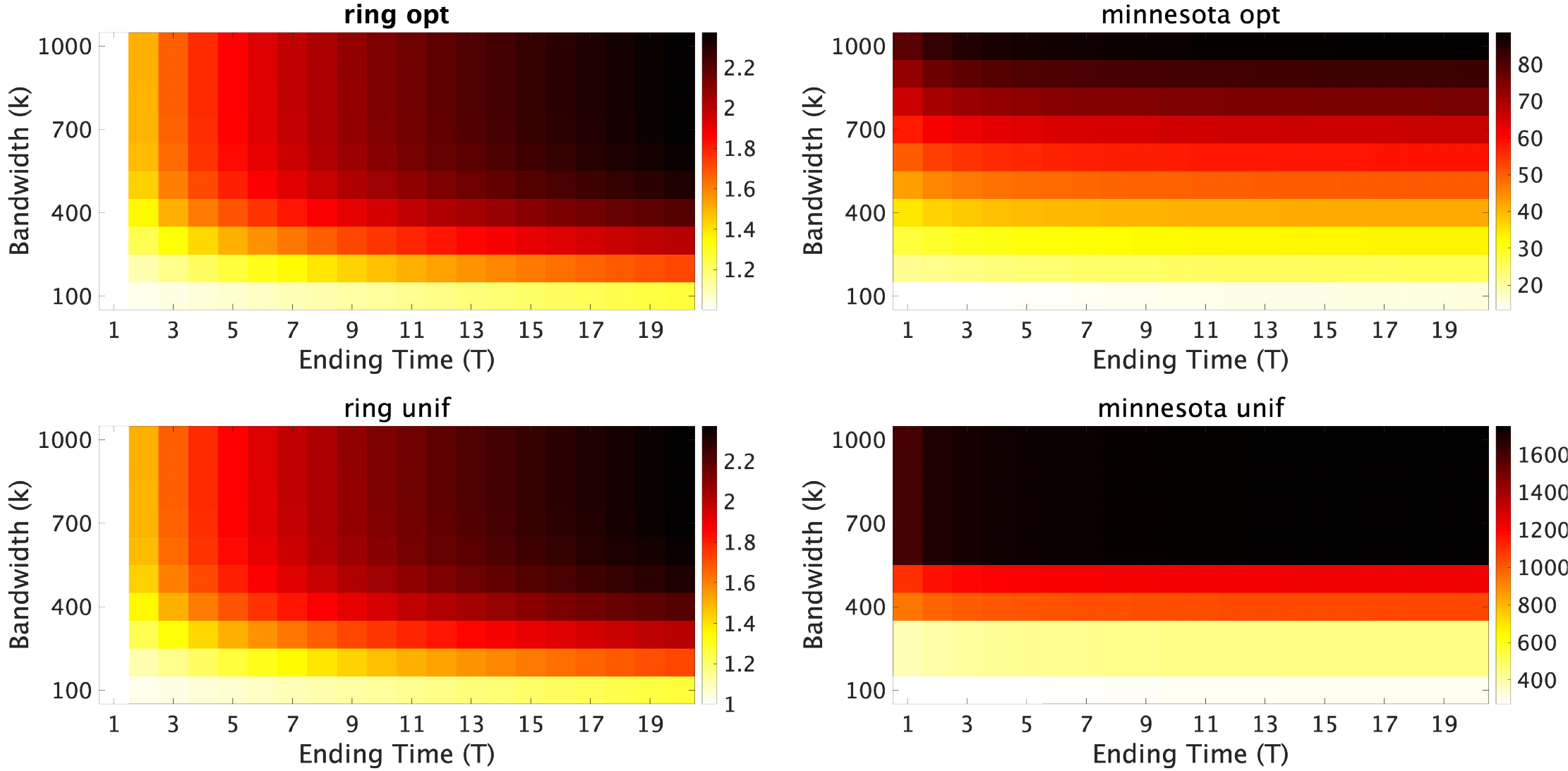}
    \caption{$\sum_{t=0}^{T-1}\nu^2_{k,T}(t)$ plot with various $T=1:1:20$ and $k=100:100:1000$ for ring graph (left) and minnesota graph(right). Top is for optimal sampling and bottom is for uniform sampling.}
    \label{tab:coh_m}
\end{figure}

In \Cref{tab:coh_r}, the coherence values of both optimal and uniform sampling distributions for the ring graph remain approximately the same at each time step. This can be seen from our formula  in   Equation \eqref{optcoherence} and arguments in the remark of  \Cref{prop:optimal}.  For the Minnesota graph, unlike the ring case, the coherence values differ significantly between the two sampling methods.  However, the trends of $\nu^2_{k,T}(t)$ on $t$ under both distributions on both graphs indicate a greater proportion of sampling in the earlier time steps and a lower proportion in the later time steps. This is reasonable as signals at earlier times are informative than signals at later times which are very smooth due to the diffusion process. 

{In \Cref{tab:coh_m}, we compute the sum of the coherence parameters at each time step for various $(T,k)$. Darker colors indicate higher values, suggesting a greater number of observed space-time samples needed. Following the fixed $(T,k)$ experiment, the sum at each $(T,k)$ remains very close for both distributions on the ring graph. Additionally, for ring graph with $T= 1$ on plot, they are the same, which aligns with our theoretical analysis. Moreover, with increasing bandwidth and a larger time span, more observed samples are required. This makes sense since higher bandwidth and larger time span bring more complexity. 

On Minnesota graph, $\sum_{t=0}^{T-1}\nu^2_{k,T}(t)$ slightly depends on the time length and is mainly influenced by the signal bandwidth, which suggests a similar level of signal complexity over time on irregular graph. Unlike the ring graph, the coherence values for the uniform distribution are not only much higher than those of the optimal distribution but also change more dramatically than those of the optimal distribution when bandwidth varies, which makes optimal sampling distribution much needed in such scenarios. }

From the experiments, we observe that the gap in incoherence between uniform distributions and optimal distributions for regular graphs is smaller compared to irregular graphs. This can be explained by our theoretical derivation of coherence formulas: the eigenvectors for regular graphs have uniform magnitudes, making the optimal distribution closer to the uniform distribution. Additionally, our constructions of optimal sampling distributions take graph topology into account and the resulting sampling locations would yield sub-optimal efficiency for both irregular and regular graphs.

\subsection{Noiseness Data Recovery}\label{sec:noiseless}

We now apply our sampling and recovery methods on synthetically generated graph signals on the ring and Minnesota graphs. We wish to study the success rate of recovery as the number of samples   $m$ and the sparsity $s$ of the frequency domain vary. The  noiseless set-up of this section is to verify the sampling complexity of our recovery problem stated in  \Cref{mainthm}.

 Both graphs have 2642 nodes, generated using the GSP toolbox with $\{0,1\}$ valued adjacent matrix. We take {$k=1000$}, $m =$ 10:10:1000, $s =$ 1:50 and $T=1,10$.  For the ring graph, we pick the operator $\mathbf{A} =\exp(-4\mathbf{L})$, while for the Minnesota graph, we choose $\mathbf{A}=\exp(-0.5\mathbf{L})$ with $\mathbf{L}$ being the graph laplacian. This corresponds two scenarios where we have fast diffusion and slow diffusion.

    For each pair $(m,s)$, we randomly generate $100$ normalized $s$-sparse $k$-bandlimited signals. A recovery is considered successful if the relative error between the estimated and ground truth signal is within 1\%. We record the number of successful recoveries. 

In our simulation,   we generated $m_t$ from $m$ for each $t \in \{0\}\cup[T-1]$ when $T>1$ as follows. For both distributions, we determined the number of samples $m_t$ based on the coherence value, $m_t = \frac{\nu_{k,T}^2(t)}{\sum_{t=0}^{T-1}\nu_{k,T}^2(t)}m$, and then  $m_t$ samples are randomly taken from $n$ nodes according to the respective probability distribution $\mathbf{p}_t$. It is easy to see that, for uniform distributions, each time step, we uniformly choose $m_t =m/T$ samples while the optimal distributions will sample adaptively according to the graph topology and dynamics.  

The results are shown in \Cref{fig:rec_r,fig:rec_m}, where red, white, and blue indicate success rates of 0\%, 50\%, and 100\%, respectively. 

\begin{figure}
    \centering
    \begin{tabularx}{0.9\linewidth}{*{2}{X}}
    \end{tabularx}
    \includegraphics[width=0.9\linewidth]{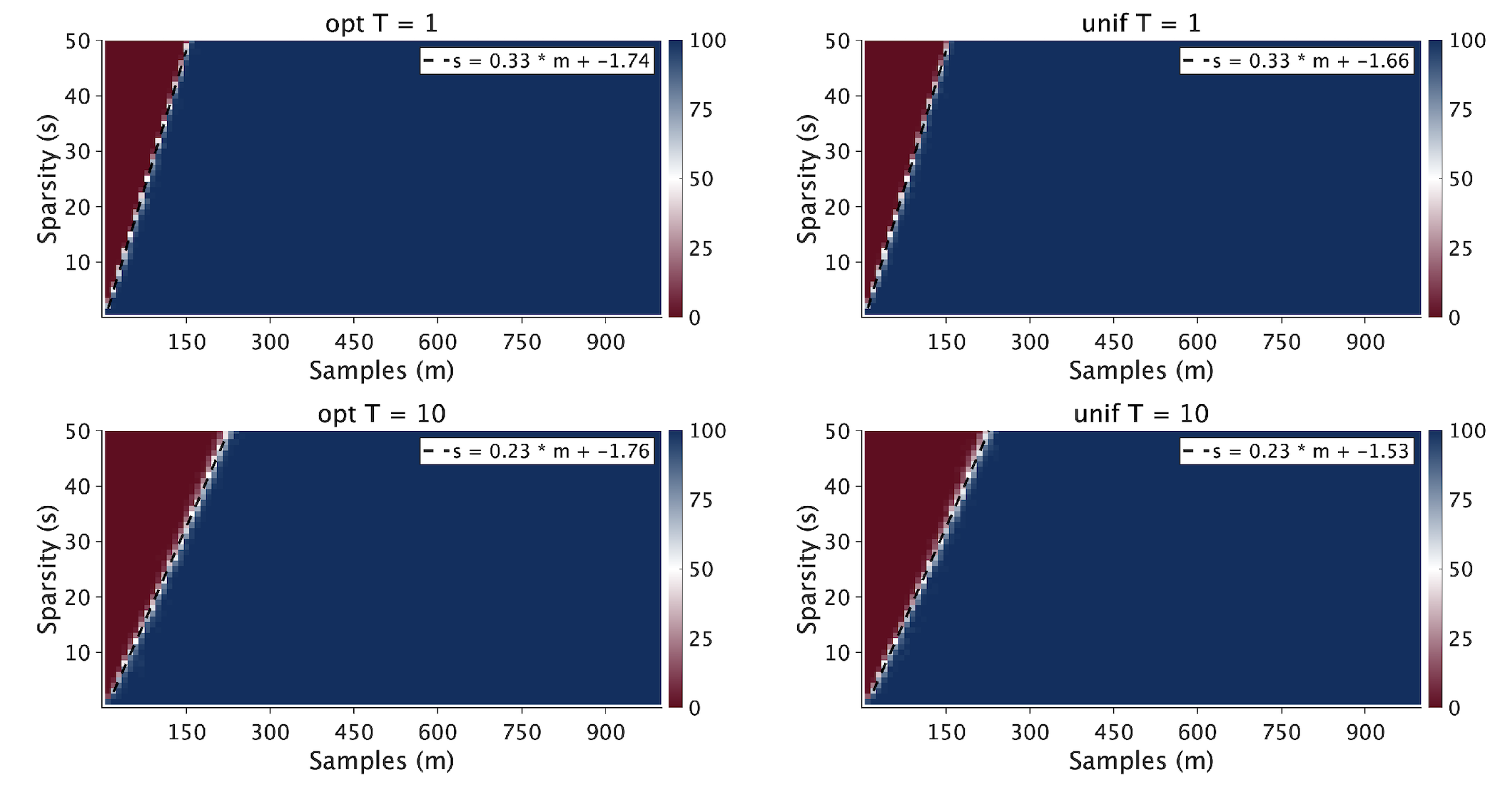}
    \caption{Phase transition diagram for the ring graph (top: $T = 1$; bottom: $T = 10$) using optimal (left) and uniform distributions (right). The slopes of boundaries from red to blue are fitted using black dot lines, indicating the critical sampling complexities that ensure  recovery with high probability. }
    \label{fig:rec_r}
\end{figure}

\begin{figure}
    \centering
    \includegraphics[width=0.9\linewidth]{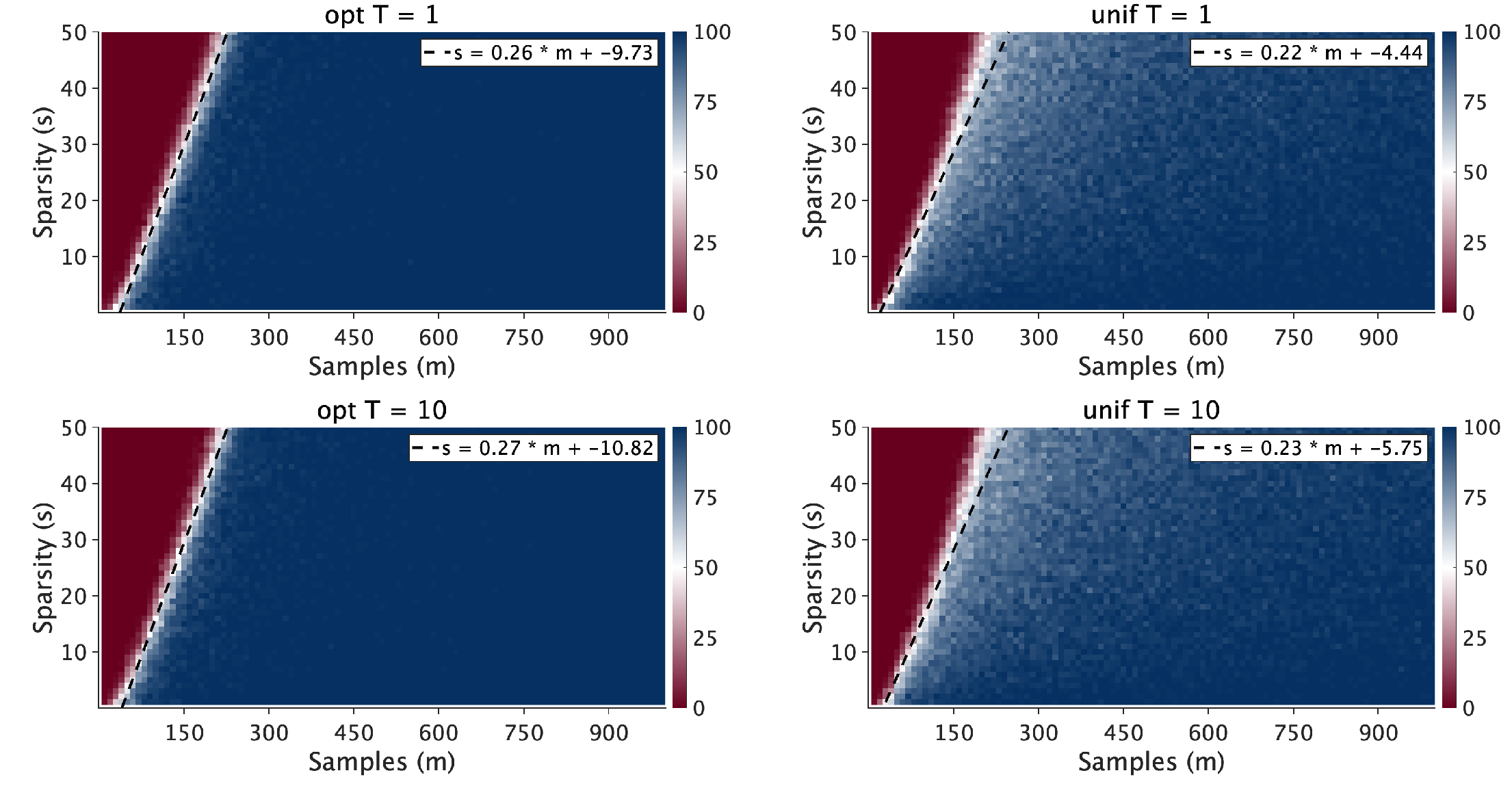}
    \caption{Phase transition diagram for the Minnesota graph (top: $T = 1$; bottom: $T = 10$) using optimal (left) and uniform distribution (right). The slopes of boundaries from red to blue are fitted using black dot lines, indicating the critical sampling complexities that ensure  recovery with high probability. }
    \label{fig:rec_m}
\end{figure}
\Cref{fig:rec_r} shows the results for ring graph with $T=1$ and $T=10$. The phase diagrams are approximately the same for both uniform and optimal sampling distributions. This aligns with our discussion in \Cref{sec:exp_coh} regarding the ring graph, where we demonstrated that the coherence values {were the same} between the two distributions.  
In \Cref{fig:rec_m}, for the Minnesota graph, uniform sampling performs  worse than optimal sampling, requiring a greater number of samples for recovery. This is consistent with the coherence values shown in \Cref{tab:coh_m}, where the uniform distribution exhibits a  higher coherence value, thereby suggesting a higher number of samples for recovery according to \eqref{ripbound3}.


In both the ring graph and Minnesota graph, we observe that under the optimal sampling distribution, the phase transition can be linearized with respect to $s$. This is consistent with the claims of \Cref{thm:RIP}, in which the number of samples required for RIP is proportional to $s$. Further, under the optimal distribution, we can overcome the added complexity and irregularity of the Minnesota graph and have comparable recovery results to the ring graph.

\subsection{Noisy Data Recovery}\label{sec:noisyrecovery}

In this section, we evaluate the performance of our recovery algorithm in the presence of noise, modeled as:

\[
\mathbf{y} = \mathbf{S}(\pi_{\mathbf{A},T}(\mathbf{x}))+ \mathbf{e},
\] where \( \mathbf{e} \sim \mathrm{Unif}([-\sigma,\sigma]^{nT})\). 

In contrast to the noise-free case, where recovery was evaluated for sparsity levels \( s = 1:1:50 \) with \( k=1000 \), we now consider a reduced range of sparsity levels \( s = 1:1:10 \), while keeping the sampling budget constant across both settings. This adjustment is necessary because in the noisy case, stable recovery requires more samples to compensate for the effect of noise. As shown in Proposition \ref{thm:cosamp_stable}, a sufficient recovery condition in the noisy scenario is guaranteed when the Restricted Isometry Property (RIP) constant satisfies \( \delta_{8s} < 0.4782 \), compared to \( \delta_{s} < 0.4782 \) in the noise-free case, which implies a higher sampling complexity might be needed. This is confirmed by our numerical experiments later. 

For each sparsity level \( s \), we randomly generate sparse signals \( \mathbf{x} \)  and its noisy  version ($\sigma =10^{-3}$), and then test recovery with a sampling budget of \( m = 100:100:1000 \), conducting 100 random trials using both uniform and optimal sampling distributions. As suggested by Proposition \ref{thm:cosamp_stable}, we define:

\[
\text{error ratio} =\frac{\|\mathbf{x} - \hat{\mathbf{x}}\|_2}{\|\mathbf{\Psi} \mathbf{e}\|_2\|\mathbf{x}\|_2 }.
\]

In the experiments, we calculate the mean of the middle 90\% of the error metrics, which falls between the 5th and 95th percentile.

\begin{figure}[H]
    \centering
    \begin{tabularx}{0.9\linewidth}{*{2}{X}}
    \end{tabularx}
    \includegraphics[width=0.8\linewidth]{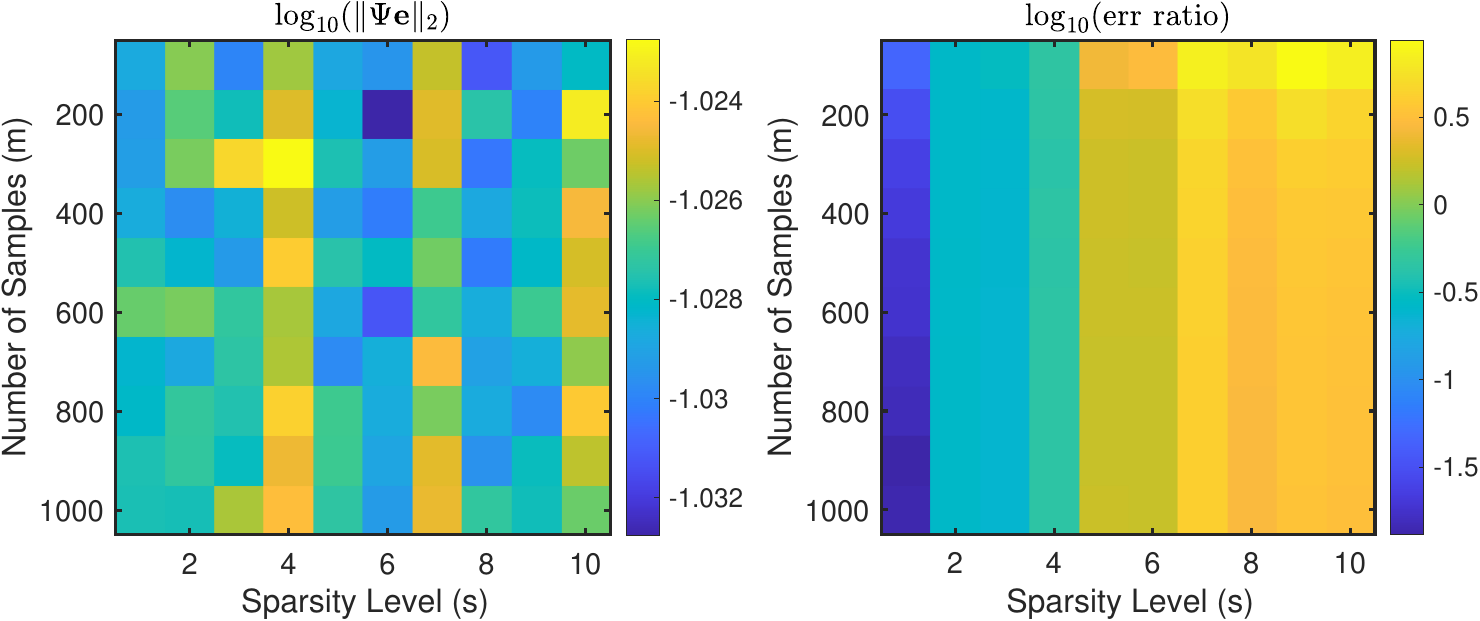}
    \caption{Error plots for ring graph with $T=10$ using uniform (optimal) distribution. The noise perturbation $\|\mathbf{\Psi} \mathbf{e}\|_2$ is centred around $ 10^{-1}$. The best achievable best errors range falls into  $[10^{-2.9}, 10^{-0.5}]$ as $s$ increases from 1 to 10.}
    \label{fig:rec_r_noisy}
\end{figure}

\begin{figure}[H]
    \centering
    \begin{tabularx}{0.9\linewidth}{*{2}{X}}
    \end{tabularx}
    \includegraphics[width=0.8\linewidth]{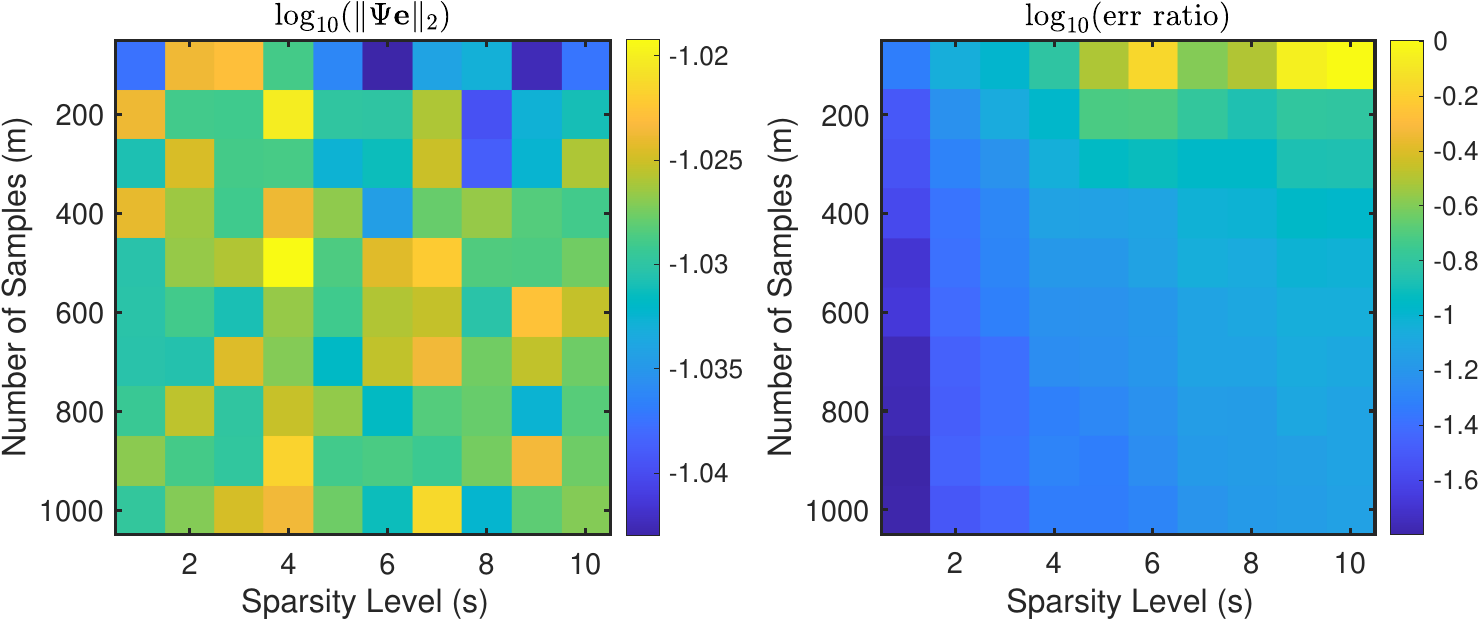}
     \includegraphics[width=0.8\linewidth]{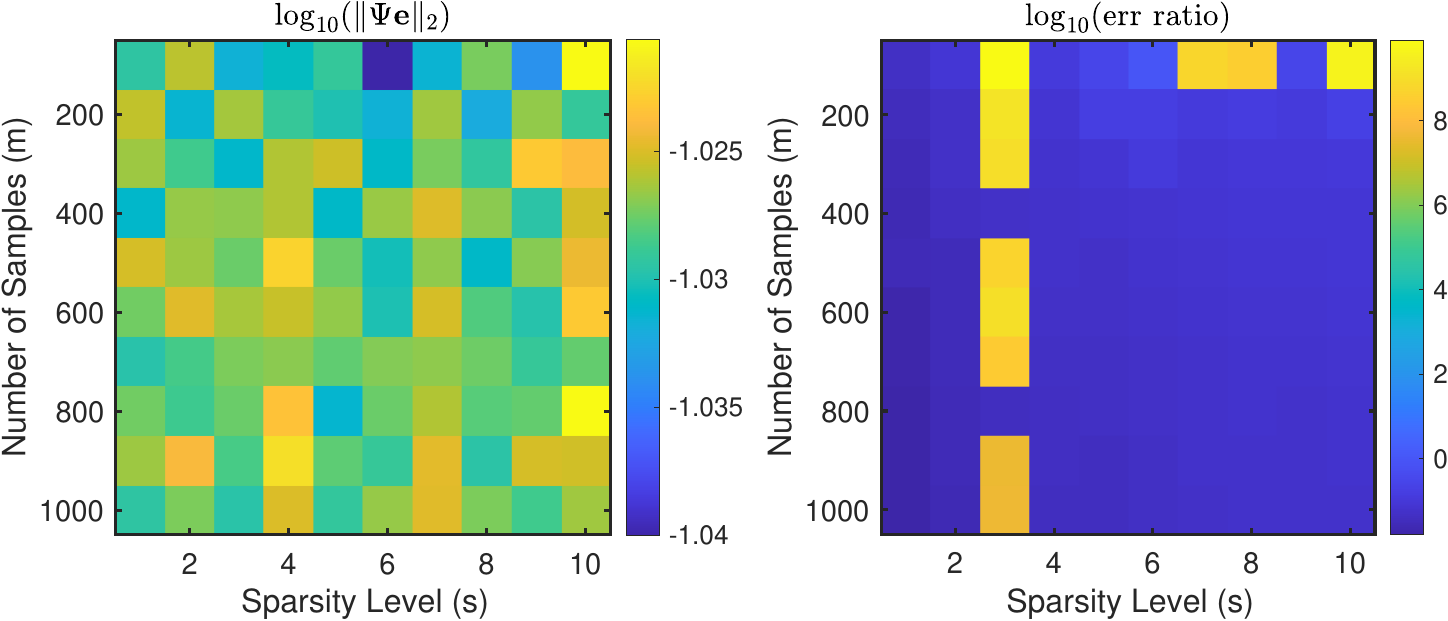}
    \caption{Error plots for Minnesota graph with $T=10$ using optimal distribution (top) and uniform distribution (bottom). The noise perturbation $\|\mathbf{\Psi} \mathbf{e}\|_2$ is centred $10^{-1}$. The best achievable best errors range falls into  $[10^{-2.9}, 10^{-2.13}]$ as $s$ increases from 1 to 10 for optimal distribution. The uniform distribution had similar recovery accuracy but with failed recovery for the sparsity $s=3$. }
    \label{fig:rec_r_noisy}
\end{figure}

We summarize our findings as follows:

\begin{itemize}

\item The error ratio increases as \(s\) increases, as predicted by Proposition \ref{thm:cosamp_stable}, since the constant \(D\) depends on \(\delta_{8s}\), which grows with \(s\).

\item For each sparsity level, there is a limit to recovery accuracy, achievable with a relatively small number of samples. In both examples, we observe that good recovery accuracy is attained with \(m = 300\) samples, and further increasing the number of samples does not significantly improve accuracy.

\item Optimal sampling distributions exhibit more stable behavior compared to uniform distributions. While uniform sampling may fail to accurately estimate some sparse signals, optimal sampling accounts for all sparse signals and results in more uniform recovery performance.

\item The robustness can be further enhanced by employing alternative $\ell_1$ 
 -minimization techniques, such as proximal gradient methods and ADMM. We leave this exploration for future work.

\end{itemize}

\FloatBarrier

\begin{appendices}
\section{Proofs and Supplementary Materials}\label{sec:appendix}
\begin{theorem}\label{thm:RIP}
Given $\eta, \beta, \epsilon \in (0,1)$, the sampling operator $\mathbf{\Psi} \mathbf{S} \tilde{\mathbf{U}}_{k,T}$ possesses a restricted isometry constant $\delta_s$ that is bounded by $\eta+\eta^2+\beta$. This is achieved with a probability of at least $1 - \epsilon$ under the following conditions: for every iteration $t = 0, 1, \ldots, T-1$,
\begin{align}
m_t\geq& \frac{32s(\nu_{k,t})^2}{3\beta^2}\log(\epsilon^{-1})\label{ripbound1}\\
m_t\geq& \frac{6272s(\nu_{k,t})^2 \log(3nT)\log(4k)\log^2(4s)}{\eta^2}\label{ripbound2}.
\end{align}  where $m_t$ is the number of samples taken at time $t$.
\end{theorem}

\subsection{Proof of \Cref{thm:RIP}}
Before we proceed to provide the proof for \Cref{thm:RIP}, let's introduce some notations and  some supporting results.   Let $\mathbf{x}_1,\cdots,\mathbf{x}_M$ be vectors in $\mathbb{R}^{k}$. Assume that  $K:=\max_{\ell\in[M]}\max_{\mathcal{S}\subseteq[k]}\|(\mathbf{x}_\ell)_{\mathcal{S}}\|$  and $K_{\infty}:=\max_{\ell\in[M]}\|\mathbf{x}_{\ell}\|_{\infty}$. Let's introduce the set 
\begin{equation}\label{eqn:S-sparse-sets}
D_{s,k}:=\{\mathbf{z}\in\mathbb{R}^{k}:\|\mathbf{z}\|_2\leq 1,\|\mathbf{z}\|_0\leq s\}=\bigcup_{\mathcal{S}\subseteq[k],|\mathcal{S}|=s} B_\mathcal{S},
\end{equation} 
where $B_\mathcal{S}$ denotes the unit sphere in $\mathbb{R}^\mathcal{S}$ with respect to the $\ell_2$-norm. Also, let's introduce the seminorm 
\[\interleave \mathbf{B}\interleave_{s}:=\sup_{\mathbf{z}\in D_{s,k}}|\langle \mathbf{B}\mathbf{z},\mathbf{z}\rangle|
\]
 on $\mathbb{R}^{k\times k}$ and the norm 
\begin{eqnarray}\label{eqn:seminorm}
    \|\mathbf{u}\|_{X}:=\max_{\ell\in[M]}|\langle \mathbf{x}_{\ell},\mathbf{u}\rangle|,\mathbf{u}\in\mathbb{R}^k. 
\end{eqnarray}
\begin{lemma}[{\cite[Lemma 12.37]{Foucart2013-sv}}]\label{lmm:coveringnumber}
    Let $U$ be a subset of $B_{\|\cdot\|_1}^k:=\{\mathbf{x}\in\mathbb{R}^k,\|\mathbf{x}\|_1\leq 1\}$ and $\|\cdot\|_{X}$ be the seminorm in \eqref{eqn:seminorm} defined for vectors $\mathbf{x}_{\ell}$ with $\|\mathbf{x}_{\ell}\|_{\infty}\leq K_{\infty}$ for all $\ell\in[M]$. 
Then for $0<t\leq 3K_{\infty}$,
    \[
    \log(2\mathcal{N}(U,\|\cdot\|_{X},t))\leq 6K_{\infty}\sqrt{\log(3M)\log(4k)}/t.
    \]
\end{lemma}
\begin{prop}[{\cite[Proposition C.3]{Foucart2013-sv}}]\label{prop:covingnumber}
    Let $\|\cdot\|$ be some norm on $\mathbb{R}^k$ and let U be a subset of the unit ball $B=\{\mathbf{x}\in\mathbb{R}^k:\|\mathbf{x}\|\leq 1\}$. Then the  covering number satisfies, for $t>0$,
    \[
    \mathcal{N}(U,\|\cdot\|,t)\leq(1+2/t)^{k}.
    \]
\end{prop}
\begin{theorem}[Dudley's inequality, {\cite[Theorem 8.23]{Foucart2013-sv}}]
   Let $X_t$, $t \in T$, be a centered subgaussian process with associated pseudometric $d$. Then, for any $t_0 \in T$, 
   \begin{equation}
       \mathbb{E} \sup_{t\in T} X_t \leq 4\sqrt{2} \int_{0}^{\Delta(T)/2}\sqrt{\log(\mathcal{N}(T,d,u))} du,
   \end{equation}
   \begin{equation}
       \mathbb{E} \sup_{t\in T}|X_t| \leq 4\sqrt{2} \int_{0}^{\Delta(T)/2}\sqrt{\log(2\mathcal{N}(T,d,u))} du,
   \end{equation}
   where $\Delta(T)=\sup_{t\in T} \sqrt{\mathbb{E}|X_t|^2}$ and $\mathcal{N}(T,d,u)$ denotes the covering number defined as the smallest integer $N$ such that there exists a subset $F$ of $T$ with $|F|=N$ and $\min_{s\in F} d(t, s) \leq u$ for all $t \in T $. 
\end{theorem}
\begin{lemma}\label{lmm:rademacher_est}
    Let $\mathbf{x}_1,\cdots,\mathbf{x}_M$ be vectors in $\mathbb{R}^{k}$ with $\max_{\ell\in[M]}\max_{\mathcal{S}\subseteq[k]}\|(\mathbf{x}_\ell)_{\mathcal{S}}\|\leq K $ and $\max_{\ell\in[M]}\|\mathbf{x}_\ell\|_{\infty}\leq K_{\infty}$.  Suppose $\left(\varepsilon_{\ell}\right)_{\ell=1}^M$ is a Rademacher sequence.  Then, for $s\leq M$,
\[
\mathbb{E}\interleave \sum_{\ell=1}^M\varepsilon_{\ell}\mathbf{x}_{\ell}\mathbf{x}_{\ell}^\top\interleave_s\leq 28\sqrt{2}R\sqrt{s}K_{\infty}\sqrt{\log(3M)\log(4k)}\log(4s), 
\]
where $R= \sqrt{\interleave \sum_{\ell=1}^M\mathbf{x}_{\ell}\mathbf{x}_{\ell}^\top\interleave_{s}}$.
\end{lemma}
\begin{proof}
    Observe that 
    \[
    E:=\mathbb{E}\interleave \sum_{\ell=1}^M\varepsilon_{\ell}\mathbf{x}_{\ell}\mathbf{x}_{\ell}^\top\interleave_s=\mathbb{E}\sup_{\mathbf{u}\in D_{s,k}}\left| \sum_{\ell=1}^M\varepsilon_{\ell}\langle \mathbf{x}_{\ell},\mathbf{u}\rangle^2\right|.
    \]
This is the supremum of the Rademacher process $X_{u}=\sum_{\ell=1}^M\varepsilon_{\ell}\langle \mathbf{x}_{\ell},\mathbf{u}\rangle^2$, which has associated pseudometric
\[
d(\mathbf{u},\mathbf{v})=(\mathbb{E}|X_{\mathbf{u}}-X_{\mathbf{v}}|^2)^{1/2}=\sqrt{\sum_{\ell=1}^M\left(\langle \mathbf{x}_{\ell},\mathbf{u}\rangle^2-\langle \mathbf{x}_{\ell},\mathbf{v}\rangle^2\right)^2}.
\]
Then, for $\mathbf{u},\mathbf{v}\in D_{s,k}$, the triangle inequality gives
\begin{eqnarray*}
    d(\mathbf{u},\mathbf{v})&=&\sqrt{\sum_{\ell=1}^M\left(\langle \mathbf{x}_{\ell},\mathbf{u}\rangle-\langle \mathbf{x}_{\ell},\mathbf{v}\rangle\right)^2\left(\langle \mathbf{x}_{\ell},\mathbf{u}\rangle+\langle \mathbf{x}_{\ell},\mathbf{v}\rangle\right)^2}\\
    &\leq& \sqrt{\sum_{\ell=1}^M\langle \mathbf{x}_{\ell},\mathbf{u}+\mathbf{v} \rangle^2}\max_{\ell\in[M]}\left| \langle \mathbf{x}_{\ell},\mathbf{u}-\mathbf{v} \rangle\right|\\
    &\leq& \sqrt{\sum_{\ell=1}^M(2\langle \mathbf{x}_{\ell},\mathbf{u}\rangle^2+\langle \mathbf{x}_{\ell},\mathbf{v} \rangle^2)}\max_{\ell\in[M]}\left| \langle \mathbf{x}_{\ell},\mathbf{u}-\mathbf{v} \rangle\right|\leq 2R\max_{\ell\in[M]}\left| \langle \mathbf{x}_{\ell},\mathbf{u}-\mathbf{v} \rangle\right|,
\end{eqnarray*}
where $R=\sup_{\tilde{\mathbf{u}}\in D_{s,k}} \sqrt{\sum_{\ell=1}^M\langle \mathbf{x}_{\ell},\tilde{\mathbf{u}}\rangle^2}=\sqrt{\interleave\sum_{\ell=1}^Mx_{\ell}\mathbf{x}_{\ell}^\top\interleave_s}$.  
Therefore, we have
\[
d(\mathbf{u},\mathbf{v})/(2R)\leq \|\mathbf{u}-\mathbf{v}\|_{X}.
\]
According to the Dudley's inequality, we have that
\begin{align*}
   E=\mathbb{E}\sup_{\mathbf{u}\in D_{s,k}}\left|X_{\mathbf{u}}\right|&\leq 4\sqrt{2}\int_{0}^{\Delta(D_{s,k})/2}\sqrt{\log(2\mathcal{N}(D_{s,k},d,t))}dt\\
\leq&8\sqrt{2}R\int_{0}^{\Delta/4}\sqrt{\log(2\mathcal{N}(D_{s,k},\|\cdot\|_{X},t))}dt
\end{align*}
where $\Delta(D_{s,k})=\sup_{\mathbf{u}\in D_{s,k}}\sqrt{\mathbb{E}|X_{\mathbf{u}}|^2}\leq \sup_{u\in D_{s,k}}\sqrt{\sum_{\ell=1}^M\langle \mathbf{x}_{\ell},\mathbf{u}\rangle^4}\leq R\sup_{\mathbf{u}\in D_{s,k}}\|\mathbf{u}\|_{X}=:R \Delta$.

By the Cauchy-Schwarz inequality, for $\mathbf{u}\in D_{s,k}$,
\[
\|\mathbf{u}\|_{X}=\max_{\ell\in[M]}|\langle \mathbf{x}_{\ell},\mathbf{u}\rangle|\leq \max_{\ell\in[M]}\max_{\mathcal{S}\subseteq[k]}\|(\mathbf{x}_{\ell})_{\mathcal{S}}\|_2.
\]
Therefore, $\Delta\leq \max_{\ell\in[M]}\max_{\mathcal{S}\subseteq[k]}\|(\mathbf{x}_{\ell})_{\mathcal{S}}\|_2=:K
$. Since $\|(\mathbf{x}_{\ell})_{\mathcal{S}}\|_2\leq\sqrt{s}\|\mathbf{x}_{\ell}\|_{\infty}\leq\sqrt{s}K_{\infty}$, we also have $\Delta\leq\sqrt{s}K_{\infty}$. 

Next, our goal is to estimate the covering numbers $\mathcal{N}(D_{s,k},\|\cdot\|_{X},t)$. We do this in two different ways. One estimate is good for small values of $t$ and the other one for large values of $t$. For small values of $t$, we use \Cref{prop:covingnumber}. 
Note that $\|\mathbf{x}\|_{X}\leq K\|\mathbf{x}\|_2$. Then using the subadditivity of the covering numbers, we have
\begin{align*}
  \mathcal{N}(D_{s,k},\|\cdot\|_{X},t)\leq&\sum_{\mathcal{S}\subseteq[N],|\mathcal{S}|=s}\mathcal{N}(B_{\mathcal{S}},K\|\cdot\|_2,t)\\
  =&\sum_{\mathcal{S}\subseteq[N],|\mathcal{S}|=s}\mathcal{N}(B_{\mathcal{S}},\|\cdot\|_2,t/K)\leq\binom{k}{s}\left(1+\frac{2K}{t}\right)^{s}\\
  \leq&\left(\frac{ek}{s}\right)^s\left(1+\frac{2K}{t}\right)^{s}.
\end{align*}
Therefore, for $t>0$, 
\[
\sqrt{\log(2\mathcal{N}(D_{s,k},\|\cdot\|_{X},t))}\leq \sqrt{s}\sqrt{\log({ek}/{s})+\log(1+{2K}/{t})}\leq \sqrt{s}(\sqrt{\log({ek}/{s})}+\sqrt{\log(1+{2K}/{t})}).
\] 
Additionally, notice that \[
D_{s,k}\subseteq \sqrt{s}B_{\|\cdot\|_1}^{k}, \text{ where }B_{\|\cdot\|_1}^{k}:= \{x\in\mathbb{R}^{k},\|x\|_1\leq 1\}. 
\]
According to \Cref{lmm:coveringnumber} and the property of covering number, we have
\begin{align*}
    \sqrt{\log(2\mathcal{N}(D_{s,k},\|\cdot\|_{X},t))}\leq &\sqrt{\log(2\mathcal{N}(\sqrt{s}B_{\|\cdot\|_1}^{k},\|\cdot\|_{X},t))}\\
    =&\sqrt{\log(2\mathcal{N}(B_{\|\cdot\|_1}^{k},\|\cdot\|_{X},t/\sqrt{s}))}\leq 6\sqrt{s}K_{\infty}\sqrt{\log(3M)\log(4k)}/t,
\end{align*}
for $0<t\leq 3\sqrt{s}K_{\infty}$. 
Then
\begin{align*}
&\int_{0}^{\Delta/4}\sqrt{\log(2\mathcal{N}(D_{s,k},\|\cdot\|_X,t))}dt\\
\leq&\int_{0}^{\kappa}\sqrt{s}(\sqrt{\log({ek}/{s})}+\sqrt{\log(1+{2K}/{t})})dt+\int_{\kappa}^{K/4} 6\sqrt{s}K_{\infty}\sqrt{\log(3M)\log(4k)}/t dt\\
\leq& \sqrt{s}\kappa\sqrt{\log(ek/s)}+\sqrt{s}\kappa \sqrt{\log(e(1+{2K}/{\kappa}))}+6\sqrt{s}K_{\infty}\sqrt{\log(3M)\log(4k)}\log(K/(4\kappa))\\
\end{align*}
where 
\begin{align*}
\int_{0}^{\kappa}\sqrt{\log(1+2K/t)}dt
    \leq& \sqrt{\int_{0}^{\kappa}dt\int_{0}^{\kappa}\log(1+2K /t)dt}\\
    =&\sqrt{2K\kappa\int_{2K/\kappa}^{\infty}y^{-2}\log(1+y)dy}\\
    =&\sqrt{2K\kappa}\sqrt{-y^{-1}\log(1+y)|_{2K/\kappa}^{\infty}+\int_{2K/\kappa}^{\infty}\frac{1}{y(1+y)}dy}\\
        \leq &\sqrt{2K\kappa}\sqrt{-y^{-1}\log(1+y)|_{2K/\kappa}^{\infty}+\int_{2K/\kappa}^{\infty}\frac{1}{y^2}dy}\\
    =&\sqrt{2K\kappa}\sqrt{\frac{\kappa}{2K}\log(1+\frac{2K}{\kappa})+\frac{\kappa}{2K} }\\
    \leq& \kappa \sqrt{\log(1+\frac{2K}{\kappa})+1}=\kappa \sqrt{\log(e(1+\frac{2K}{\kappa}))}.
\end{align*}
Let's choose $\kappa=K_{\infty}/3$, we have 
\[
\int_{0}^{\Delta/4}\sqrt{\log(2\mathcal{N}(D_{s,k},\|\cdot\|_X,t))}dt\leq \frac{7}{2}\sqrt{s}K_{\infty}\sqrt{\log(3M)\log(4k)}\log(4s).
\]
 
Therefore, 
\[
E\leq 28\sqrt{2}R\sqrt{s}K_{\infty}\sqrt{\log(3M)\log(4k)}\log(4s)
\]
\end{proof}

Now we are ready to provide the proof for \Cref{thm:RIP}.
\begin{proof}[The proof of \Cref{thm:RIP}]
$\pi_{ \mathbf{A},T}(x)\in\text{span}(\widetilde{\mathbf{U}}_{k,T})$, we thus only need to study the property of the matrix $\mathbf{P}_{\Omega} ^{-1/2}\mathbf{W}^{\frac{1}{2}}\mathbf{S}\widetilde{\mathbf{U}}_{k,T}$. We use the characterization of the restricted isometry constants, 
\begin{equation}\label{eqn:RIC4R2}
\begin{aligned}
    \delta_s=&\max_{S\subseteq[k],|S|=s}\|(\mathbf{P}_{\Omega}^{-1/2}\mathbf{W}^{\frac{1}{2}}\mathbf{S}\widetilde{\mathbf{U}}_{k,T}\mathbf{R}_S)^\top(\mathbf{P}_{\Omega}^{-1/2}\mathbf{W}^{\frac{1}{2}}\mathbf{S}\widetilde{\mathbf{U}}_{k,T}\mathbf{R}_S)-\id\|_{2}\\
    =&\interleave (\mathbf{P}_{\Omega}^{-1/2}\mathbf{W}^{\frac{1}{2}}\mathbf{S}\widetilde{\mathbf{U}}_{k,T})^\top(\mathbf{P}_{\Omega}^{-1/2}\mathbf{W}^{\frac{1}{2}}\mathbf{S}\widetilde{\mathbf{U}}_{k,T})-\mathrm{I} \interleave_s,
    \end{aligned}
\end{equation}
where $R_\mathcal{S}$ denotes the restriction matrix (of appropriate dimensions) such that $MR_\mathcal{S}$ is the restriction of any matrix $M$ to its columns indexed by $\mathcal{S}$.  

Notice that
\begin{equation*}
    \begin{aligned}
       \mathbf{X}:= &   (\mathbf{P}_{\Omega}^{-1/2}\mathbf{W}^{\frac{1}{2}}\mathbf{S}\widetilde{\mathbf{U}}_{k,T})^\top \mathbf{P}_{\Omega}^{-1/2} \mathbf{W}^{\frac{1}{2}}\mathbf{S}\widetilde{\mathbf{U}}_{k,T}\\
        =& \sum_{t=0}^{T-1}\sum_{i=1}^{m_t} \frac{(\mathbf{F}_{k,T}\mathbf{\Lambda}_k^t\mathbf{U}_k^{\top}\delta_{t,\omega_{t,i}}\delta_{t,\omega_{t,i}}^{\top} \mathbf{U}_k \mathbf{\Lambda}_k^t\mathbf{F}_{k,T}) }{m_t\mathbf{p}_{t}(\omega_{t,i})},
    \end{aligned} 
\end{equation*} 
where $\delta_{t,\omega_{t,i}}=\begin{cases}
    1,\omega_{t,i}=t\\
    0,\textnormal{otherwise}
\end{cases}$, 
    $F_{k,T} = \text{diag}(\begin{bmatrix} f_T(\lambda_1) & f_T(\lambda_2) & \cdots &f_T(\lambda_k)\end{bmatrix}) \text{ and }
    \Lambda_k = \text{diag}(\begin{bmatrix}
        \lambda_1 & \lambda_2 & \cdots & \lambda_k
    \end{bmatrix})$. 
Set $\mathbf{X}_{t,i}= \frac{\mathbf{F}_{k,T}\mathbf{\Lambda}_k^t\mathbf{U}_k^{\top}\delta_{t,\omega_{t,i}}}{\sqrt{m_t\mathbf{p}_{t}(\omega_{t,i})}}$. Then $\mathbf{X} = \sum_{t=0}^{T-1}\sum_{i=1}^{m_t} \mathbf{X}_{t,i}\mathbf{X}_{t,i}^\top$.

Thus $\mathbf{X}$ is a sum of $M:=\sum_{t=0}^{T-1}m_t$ independent, random, self-adjoint, positive semi-definite matrices.
Note that
\begin{equation*}
    \begin{aligned}
     \mathbb{E}(\mathbf{X}_{t,i}\mathbf{X}_{t,i}^\top) 
     =&\mathbb{E}\left(  \frac{\mathbf{F}_{k,T}\mathbf{\Lambda}_k^t\mathbf{U}_k^{\top}\delta_{t,\omega_{t,i}}\delta_{t,\omega_{t,i}}^{\top} \mathbf{U}_k \Lambda_k^t\mathbf{F}_{k,T} }{m_t\mathbf{p}_{t}(\omega_{t,i})}\right)\\
     =& \sum_{i=1}^{n}\mathbf{p}_{t}(i)\cdot  \frac{(\mathbf{F}_{k,T}\mathbf{\Lambda}_k^t\mathbf{U}_k^{\top}\delta_{t,i}\delta_{t,i}^{\top} \mathbf{U}_k \mathbf{\Lambda}_k^t\mathbf{F}_{k,T}) }{m_t\mathbf{p}_{t}(i)}
      =\frac{1}{m_t} \mathbf{\Lambda}_{k}^{2t}\mathbf{F}_{k,T}^2.
    \end{aligned}
\end{equation*}
Thus $\sum_{t=0}^{T-1}\sum_{i=1}^{m_t} \mathbb{E}(\mathbf{X}_{t,i}\mathbf{X}_{t,i}^\top)=
\diag(\frac{\sum_{t=0}^{T-1} \lambda_1^{2t}}{f_{T}^2(\lambda_1)},\frac{\sum_{t=0}^{T-1} \lambda_2^{2t}}{f_{T}^2(\lambda_2)},\cdots,\frac{\sum_{t=0}^{T-1}  \lambda_k^{2t}}{f_{T}^2(\lambda_k)})=\id$.
Therefore,
\[
\delta_s=\interleave \sum_{t=0}^{T-1}\sum_{i=1}^{m_t}(\mathbf{X}_{t,i}\mathbf{X}_{t,i}^\top-\mathbb{E}(\mathbf{X}_{t,i}\mathbf{X}_{t,i}^\top))\interleave_{s}.
\]
Let's first consider the expectation of $\delta_s$. Using symmetrization, we have
\[
\mathbb{E}\interleave \sum_{t=0}^{T-1}\sum_{i=1}^{m_t}(\mathbf{X}_{t,i}\mathbf{X}_{t,i}^\top-\mathbb{E}(\mathbf{X}_{t,i}\mathbf{X}_{t,i}^\top))\interleave_{s}\leq 2\mathbb{E}\interleave \sum_{t=0}^{T-1}\sum_{i=1}^{m_t}\varepsilon_{t,i}\mathbf{X}_{t,i}\mathbf{X}_{t,i}^\top\interleave_s,
\]
where $\left(\varepsilon_{t,i}\right)_{t,i}$ is a Rademacher sequence. By \Cref{lmm:rademacher_est}, 
\begin{align*}
   F:= \mathbb{E}(\delta_s)=&\mathbb{E}\interleave \sum_{t=0}^{T-1}\sum_{i=1}^{m_t}(\mathbf{X}_{t,i}\mathbf{X}_{t,i}^\top-\mathbb{E}(\mathbf{X}_{t,i}\mathbf{X}_{t,i}^\top))\interleave_{s}\\
\leq&2\mathbb{E}_{X}\mathbb{E}_{\varepsilon}\interleave \sum_{t=0}^{T-1}\sum_{i=1}^{m_t}(\varepsilon_{t,i}\mathbf{X}_{t,i}\mathbf{X}_{t,i}^\top))\interleave_{s}\\
\leq&56\sqrt{2}\sqrt{s}K_{\infty}\sqrt{\log(3M)\log(4k)}\log(4s)\mathbb{E}_{\mathbf{X}}\sqrt{ \interleave \sum_{t=0}^{T-1}\sum_{i=1}^{m_t}(\mathbf{X}_{t,i}\mathbf{X}_{t,i}^\top)\interleave_s}\\
\leq&56\sqrt{2}\sqrt{s}K_{\infty}\sqrt{\log(3M)\log(4k)}\log(4s)\mathbb{E}_{X}\sqrt{ \interleave \sum_{t=0}^{T-1}\sum_{i=1}^{m_t}(\mathbf{X}_{t,i}\mathbf{X}_{t,i}^\top)-\id\interleave_s+1}\\
\leq &56\sqrt{2}\sqrt{s}K_{\infty}\sqrt{\log(3M)\log(4k)}\log(4s)\sqrt{ \mathbb{E}_{\mathbf{X}}\interleave \sum_{t=0}^{T-1}\sum_{i=1}^{m_t}(\mathbf{X}_{t,i}\mathbf{X}_{t,i}^\top)-\id\interleave_s+1}\\
=&56\sqrt{2}\sqrt{s}K_{\infty}\sqrt{\log(3M)\log(4k)}\log(4s)\sqrt{F+1},
\end{align*}
where $M=\sum_{t=0}^{T-1}m_t$. 
Therefore, 
\[
F\leq D^2/2+D\sqrt{1+D^2/4}\leq D^2+D, \textnormal{ where }D=56\sqrt{2}\sqrt{s}K_{\infty}\sqrt{\log(3M)\log(4k)}\log(4s).
\]
If $D\leq\eta$ for some $\eta\in(0,1)$, we thus have $F\leq\eta+\eta^2$. Note that if we have  
\[
\|X_{t,i}\|_{\infty}=\frac{1}{\sqrt{m_t}}\left\|\frac{F_{k,T}\Lambda_{k}^tU_{k}^\top\delta_{t,\omega_{t,i}}}{\sqrt{\mathbf{p}_t(\omega_{t,i})}}\right\|_{\infty}=\frac{1}{\sqrt{m_t}}\nu_{k,t}\leq \frac{\eta}{56\sqrt{2}\sqrt{s}\sqrt{\log(3M)\log(4k)}\log(4s)},
\]
i.e., 
\begin{equation}\label{eqn:rest1_mt}
    m_t\geq \frac{6272(\nu_{k,t} )^2}{\eta^2}s\log(3nT)\log(4k)\log^2(4s), 
\end{equation} then 
 $D\leq\eta$ will hold. 

Next let's show that $\delta_s$ does not deviate much from its expectation with high probability. To this end, we use the deviation inequality of 
\begin{theorem}
    Let $\mathcal{F}$ be a countable set of functions $F:\mathbb{R}^n\rightarrow R$. Let $\mathbf{Y}_1,\cdots,\mathbf{Y}_M$ be independent random vectors in $\mathbb{R}^n$ such that $\mathbb{E}(F(\mathbf{Y}_{\ell}))=0$ and $F(\mathbf{Y}_{\ell})\leq K$ almost surely for all $\ell\in[M]$ and for all $F\in\mathcal{F}$ for some constant $K>0$. Introduce
    \[Z=\sup_{F\in\mathcal{F}}\sum_{\ell=1}^{M}F(\mathbf{Y}_{\ell}).\]
    Let $\sigma_{\ell}^2>0$ such that $\mathbb{E}(F(\mathbf{Y}_{\ell})^2)\leq\sigma_{\ell}^2$ for all $F\in\mathcal{F}$ and $\ell\in[M]$. Then for all $t>0$,
    \begin{equation}
       \label{eqn:deviation-ineqaulity}
       \mathbb{P}(Z\geq \mathbb{E}Z+t)\leq\exp\left( -\frac{t^2/2}{\sigma^2+2K\mathbb{E}(Z)+tK/3}\right)
    \end{equation}
    where $\sigma^2=\sum_{\ell=1}^M\sigma_{\ell}^2$.
\end{theorem}
By definition of the norm $\interleave\cdot\interleave_s$, we can write
\begin{align*}
    \delta_s=&\interleave \sum_{t=0}^{T-1}\sum_{i=1}^{m_t}(\mathbf{X}_{t,i}\mathbf{X}_{t,i}^\top-\mathbb{E}(\mathbf{X}_{t,i}\mathbf{X}_{t,i}^\top))\interleave_{s}\\
    =&\sup_{\mathcal{S}\subseteq[k],|\mathcal{S}|=s}\|\sum_{t=0}^{T-1}\sum_{i=1}^{m_t}((\mathbf{X}_{t,i})_{\mathcal{S}}(\mathbf{X}_{t,i})_\mathcal{S}^\top-\mathbb{E}((\mathbf{X}_{t,i})_{\mathcal{S}}(\mathbf{X}_{t,i})_{\mathcal{S}}^\top))\|\\
    =&\sup_{(\mathbf{z},\mathbf{w})\in Q_{s,k}}\langle \sum_{t=0}^{T-1}\sum_{i=1}^{m_t}(\mathbf{X}_{t,i}\mathbf{X}_{t,i}^\top-\mathbb{E}(\mathbf{X}_{t,i}\mathbf{X}_{t,i}^\top))\mathbf{z},\mathbf{w}\rangle\\
    =&\sup_{(\mathbf{z},\mathbf{w})\in Q_{s,k}^*}\sum_{t=0}^{T-1}\sum_{i=1}^{m_t}\langle (\mathbf{X}_{t,i}\mathbf{X}_{t,i}^\top-\mathbb{E}(\mathbf{X}_{t,i}\mathbf{X}_{t,i}^\top))\mathbf{z},\mathbf{w}\rangle,
\end{align*}
where $(\mathbf{X}_{t,i})_{\mathcal{S}}$ denotes the vector $\mathbf{X}_{t,i}$ restricted to the entries in $\mathcal{S}$ and $Q_{s,k}=\bigcup_{\mathcal{S}\subseteq[k],|\mathcal{S}|\leq s}Q_{\mathcal{S},k}$ with
\[
Q_{\mathcal{S},k}=\{(\mathbf{z},\mathbf{w}):\mathbf{z},\mathbf{w}\in\mathbb{R}^k,\|\mathbf{z}\|=\|\mathbf{w}\|=1,\textnormal{supp}(\mathbf{z}),\textnormal{supp}(\mathbf{w})\subseteq \mathcal{S}\}.
\]
$Q_{s,k}^*$ denotes a densse countable subset of $Q_{s,k}$. Introducing $f_{\mathbf{z},\mathbf{w}}(\mathbf{X})=\langle (\mathbf{X}\mathbf{X}^\top-\mathbb{E}(\mathbf{X}\mathbf{X}^\top))\mathbf{z},\mathbf{w}\rangle$, we therefore have
\[
\delta_{s}=\sup_{(\mathbf{z},\mathbf{w})\in Q_{s,k}^*}\sum_{t=0}^{T-1}\sum_{i=1}^{m_t}f_{\mathbf{z},\mathbf{w}}(\mathbf{X}_{t,i}).
\]

To check the boundedness of $f_{\mathbf{z},\mathbf{w}}$ for $(\mathbf{z},\mathbf{w})\in Q_{\mathcal{S},k}$ with $|\mathcal{S}|=s$, we notice that
\begin{align*}
    |f_{\mathbf{z},\mathbf{w}}(\mathbf{X}_{t,i})|=&|\langle (\mathbf{X}_{t,i}\mathbf{X}_{t,i}^\top-\mathbb{E}(\mathbf{X}_{t,i}\mathbf{X}_{t,i}^\top))\mathbf{z},\mathbf{w}\rangle|\leq \|\mathbf{z}\|_2\|\mathbf{w}\|_2\|(\mathbf{X}_{t,i})_{\mathcal{S}}(\mathbf{X}_{t,i})_\mathcal{S}^\top-\mathbb{E}((\mathbf{X}_{t,i})_{\mathcal{S}}(\mathbf{X}_{t,i})_{\mathcal{S}}^\top)\|\\
    \leq&\|(\mathbf{X}_{t,i})_{\mathcal{S}}(\mathbf{X}_{t,i})_\mathcal{S}^\top-\mathbb{E}((\mathbf{X}_{t,i})_{\mathcal{S}}(\mathbf{X}_{t,i})_{\mathcal{S}}^\top)\|_{1}\\
    =&\max_{j\in \mathcal{S}}\sum_{p\in \mathcal{S}} \left| \frac{\lambda_j^t\lambda_p^t}{m_t\mathbf{p}_{t}( {\omega_{t,i}})f(\lambda_j)f(\lambda_p)}\mathbf{u}_j(i)\mathbf{u}_p(i)-\delta_{j,p}\frac{\lambda_j^t\lambda_p^t}{m_tf(\lambda_j)f(\lambda_p)}\right|\leq sK_{\infty}^{2},
\end{align*}
where $\delta_{j,p}=\begin{cases}
    1, j=p\\
    0, \textnormal{otherwise}
\end{cases}$.

For the variance term, we estimate
\begin{align*}
\mathbb{E}(f_{\mathbf{z},\mathbf{w}}(\mathbf{X}_{t,i}))^2\leq&\mathbb{E}\langle (\mathbf{X}_{t,i}\mathbf{X}_{t,i}^\top-\mathbb{E}(\mathbf{X}_{t,i}\mathbf{X}_{t,i}^\top))\mathbf{z},\mathbf{w}\rangle^2\\
=&\mathbb{E} w^\top (\mathbf{X}_{t,i}\mathbf{X}_{t,i}^\top-\mathbb{E}(\mathbf{X}_{t,i}\mathbf{X}_{t,i}^\top))\mathbf{z}\mathbf{z}^\top (\mathbf{X}_{t,i}\mathbf{X}_{t,i}^\top-\mathbb{E}(\mathbf{X}_{t,i}\mathbf{X}_{t,i}^\top))\mathbf{w}^\top\\
\leq&\|\mathbf{w}\|^2\mathbb{E}\|(\mathbf{X}_{t,i}\mathbf{X}_{t,i}^\top-\mathbb{E}(\mathbf{X}_{t,i}\mathbf{X}_{t,i}^\top))\mathbf{z}\mathbf{z}^\top (\mathbf{X}_{t,i}\mathbf{X}_{t,i}^\top-\mathbb{E}(\mathbf{X}_{t,i}\mathbf{X}_{t,i}^\top))\|\\
\leq &\mathbb{E}\|(\mathbf{X}_{t,i}\mathbf{X}_{t,i}^\top-\mathbb{E}(\mathbf{X}_{t,i}\mathbf{X}_{t,i}^\top))\mathbf{z}\|^2\\
=&\mathbf{z}^\top \mathbb{E}(\mathbf{X}_{t,i}\mathbf{X}_{t,i}^\top)^2\mathbf{z}-\mathbf{z}^\top(\mathbb{E}(\mathbf{X}_{t,i}\mathbf{X}_{t,i}^\top))^2\mathbf{z}\\
=&\mathbb{E}(\|(\mathbf{X}_{t,i})_{\mathcal{S}}\|^2\langle \mathbf{X}_{t,i},\mathbf{z}\rangle^2)-\mathbf{z}^\top(\mathbb{E}(\mathbf{X}_{t,i}\mathbf{X}_{t,i}^\top))^2\mathbf{z}\\
\leq& sK_{\infty}^2\mathbf{z}^\top\mathbb{E}(\mathbf{X}_{t,i}\mathbf{X}_{t,i}^\top)\mathbf{z}-\mathbf{z}^\top(\mathbb{E}(\mathbf{X}_{t,i}\mathbf{X}_{t,i}^\top)^2\mathbf{z}\\
\leq& \left\| \frac{sK_{\infty}^2}{m_t}\mathbf{\Lambda}_{k}^{2t}\mathbf{F}_{k,T}^2 -\frac{1}{m_t^2}\mathbf{\Lambda}_{k}^{4t}\mathbf{F}_{k,T}^4\right\|.
\end{align*}
Then $\sum_{t=0}^{T-1}\sum_{i=1}^{m_t}\mathbb{E}(f_{\mathbf{z},\mathbf{w}}(\mathbf{X}_{t,i}))^2=\sum_{t=0}^{T-1}\sum_{i=1}^{m_t}\left\| \frac{sK_{\infty}^2}{m_t}\mathbf{\Lambda}_{k}^{2t}\mathbf{F}_{k,T}^2 -\frac{1}{m_t^2}\mathbf{\Lambda}_{k}^{4t}\mathbf{F}_{k,T}^4\right\|\leq sK_{\infty}^2$. 
Under the condition that $D<\eta$, we have
\begin{align*}
    \mathbb{P}(\delta_s\geq\eta+\eta^2+\beta)\leq&\mathbb{P}(\delta_s\geq\mathbb{E}(\delta_s)+\beta)\\
    \leq& \exp\left(-\frac{\beta^2/2}{sK_{\infty}^2(1+2(\eta+\eta^2)+\beta/3) }\right)\\
    =&\exp\left(-\frac{\beta^2}{2sK_{\infty}^2}\frac{1}{(1+2(\eta+\eta^2)+\beta/3) }\right)\leq \exp\left(-\frac{3\beta^2}{32sK_{\infty}^2}\right).
\end{align*}
To ensure $\exp\left(-\frac{3\beta^2}{32sK_{\infty}^2}\right)<\epsilon$, we can set 
\[
\frac{m_t}{(\nu_{k,t} )^2}\geq\frac{32s}{3\beta^2}\log(\epsilon^{-1})
\]
i.e.,
\[
m_t\geq \frac{32s(\nu_{k,t} )^2}{3\beta^2}\log(\epsilon^{-1}).
\]
Taking \eqref{eqn:rest1_mt} into account, we proved that $\delta_s\leq \eta+\eta^2+\beta$ with probability at least $1-\epsilon$ provided that $m_t$ satisfies the two conditions:
\begin{align}
    m_t\geq& \frac{32s(\nu_{k,t} )^2}{3\beta^2}\log(\epsilon^{-1}),\\
      m_t\geq& \frac{6272s(\nu_{k,t} )^2}{\eta^2}\log(3nT)\log(4k)\log^2(4s).
\end{align}
\end{proof}

\end{appendices}
\bibliographystyle{plain}
\bibliography{refs}

\end{document}